\documentclass[a4paper,12pt]{amsart}
\usepackage{amssymb}

\numberwithin{equation}{section}
\newtheorem{thm}{Theorem}[section]
\newtheorem{prop}[thm]{Proposition}
\newtheorem{lem}[thm]{Lemma}
\newtheorem{cor}[thm]{Corollary}
\newtheorem{defi}[thm]{Definition}

\theoremstyle{remark}
\newtheorem{rem}[thm]{Remark}

\newcommand{\R}{\mathbb{R}}
\newcommand{\Z}{\mathbb{Z}}
\newcommand{\N}{\mathbb{N}}
\newcommand{\C}{\mathbb{C}}

\newcommand{\al}{\alpha}
\newcommand{\be}{\beta}
\newcommand{\ga}{\gamma}
\newcommand{\Ga}{\Gamma}
\newcommand{\G}{\mathcal{G}}
\newcommand{\De}{\Delta}
\newcommand{\ka}{\kappa}

\newcommand{\fy}{\varphi}
\newcommand{\e}{\varepsilon}
\newcommand{\p}{\partial}
\newcommand{\la}{\lambda}
\newcommand{\de}{\delta}
\newcommand{\ups}{\upsilon}
\renewcommand{\t}{\tau}
\newcommand{\s}{\sigma}
\renewcommand{\th}{\theta}
\newcommand{\x}{\xi}
\newcommand{\y}{\eta}

\renewcommand{\r}{\rho}
\newcommand{\T}{\mathfrak{t}}
\newcommand{\na}{\nabla}
\newcommand{\I}{\infty}
\renewcommand{\S}{\mathcal{S}}
\newcommand{\Om}{\Omega}

\newcommand{\LR}[1]{{\langle {#1} \rangle }}
\newcommand{\lec}{{\, \lesssim \, }}
\newcommand{\gec}{{\,  \gtrsim \, }}

\renewcommand{\hat}{\widehat}
\newcommand{\ti}{\widetilde}
\newcommand{\supp}{\operatorname{supp}}

\newcommand{\EQ}[1]{\begin{equation} \begin{split} #1
 \end{split} \end{equation}}
\newcommand{\pr}{\\ &}
\newcommand{\pt}{&}
\newcommand{\pq}{\quad }

\author{S. Ibrahim}
\address{Arizona State University\\ Department of Mathematics and Statistics\\
P.O. Box 871804\\  Tempe, AZ 85287-1804\\ USA}
\email{slim.ibrahim@asu.edu}
\urladdr{http://math.asu.edu/~ibrahim}
\thanks{S. I. is grateful to the Department
of Mathematics and Statistics at Arizona State University.}

\author{M. Majdoub}
\address{Faculty of Sciences of Tunis, Department of Mathematics, }
\email{Mohamed.Majdoub@fst.rnu.tn}
\thanks{M. M. is grateful to the Laboratory of
PDE and Applications at the Faculty of Sciences of Tunis}

\author{N. Masmoudi}
\address{New York University \\
The Courant Institute for Mathematical Sciences,}
\email{masmoudi@courant.nyu.edu}
\thanks{N. M is partially supported by an NSF Grant DMS-0703145}

\author{K. Nakanishi}
\address{Department of Mathematics, Kyoto University}
\email{n-kenji@math.kyoto-u.ac.jp}

\title[Energy scattering for 2D critical wave equation]
{Energy scattering for \\ the 2D critical wave equation}
\date{\today}

\begin{document}
\begin{abstract}
We investigate existence and asymptotic completeness of the wave operators
for nonlinear Klein-Gordon and Schr\"odinger equations with a defocusing exponential nonlinearity in two space dimensions.
A certain threshold is defined based on the value of the conserved Hamiltonian, below
which the exponential potential energy is dominated by the kinetic energy via a Trudinger-Moser type inequality.
 We prove that if the energy is below or equal to the critical value,
then the solution approaches a free Klein-Gordon solution at the
time infinity. The interesting feature in the critical case is that
the Strichartz estimate together with Sobolev-type inequalities can
not control the nonlinear term uniformly on each time interval, but
with constants depending on how much the solution is concentrated.
Thus we have to trace concentration of the energy along time, in
order to set up favorable nonlinear estimates, and then to implement
Bourgain's induction argument. We show the same result for the
``subcritical" nonlinear Schr\"odinger equation.
\end{abstract}


\subjclass[2000]{35L70, 35Q55, 35B40, 35B33, 37K05, 37L50}
\keywords{Nonlinear wave equation, nonlinear Schr\"odinger equation, scattering theory, Sobolev critical exponent,
Trudinger-Moser inequality}

\maketitle
\tableofcontents


\section{Introduction}
We study the scattering theory in the energy space for nonlinear
Klein-Gordon equation ({\sf NLKG}):
\begin{equation}
\label{NLKG} \left\{
\begin{aligned}
 &\ddot u - \Delta u + u +f(u)= 0, \quad u:\R^{1+2}\to\R,\\
 &u(0,x) = u_0 (x)\in H^1(\R^2),\quad
 \partial_t u(0,x)=u_1(x)\in L^2(\R^2),
\end{aligned}
\right.
\end{equation}
where the nonlinearity $f:\C\to\C$ is defined by
\begin{equation}
 f(u) = \left(e^{4\pi |u|^2}-1-4\pi |u|^2\right)u,
\end{equation}
and for nonlinear Schr\"odinger equation ({\sf NLS}):
\begin{equation}
\label{NLS} \left\{
 \begin{aligned}
  &i\dot u+\Delta u = f(u) ,\quad
   u:\R^{1+2}\to\C, \\
  &u(0,x) = u_0 (x)\in H^1(\R^2).
 \end{aligned}
\right.
\end{equation}
Problem \eqref{NLKG} has the conserved energy
\EQ{ \label{scrit}
 E(u,t) &= \int_{\R^2}\, \Big(|\dot u|^2 + |\na u|^2
   + |u|^2 + 2F(u)\Big) dx,\\
 &:= E_0(u,t)+\int_{\R^2}\,2F(u)\, dx.}
where we denote
\EQ{ \label{F}
 F(u) := \frac{1}{8\pi}\left(e^{4\pi |u|^2}-1 - 4\pi
 |u|^2-8\pi^2|u|^4\right).}
Solutions of \eqref{NLS} satisfy the conservation of mass
and Hamiltonian
\begin{equation}
\label{mass} M(u,t):=\int_{\R^2}|u|^2  dx,
\end{equation}
\begin{equation}
\label{hamil} H(u,t):=\int_{\R^2}\, \Big(|\na u|^2 +2F(u)\Big)\, dx.
\end{equation}

The exponential type nonlinearities appear in several applications, as for example the self trapped beams in plasma. (See \cite{LLT}).
From the mathematical point of view, Cazenave in \cite{Caz} considered the Schr\"odinger equation with decreasing exponential and showed the global well-posedness and scattering.
With increasing exponentials, the situation is much more complicated (since there is no a priori $L^\infty$ control of the nonlinear term).
The two dimensional case is particularly interesting because of its relation to the critical Sobolev (or Trudinger-Moser) embedding.
On the other hand, we have subtracted the cubic part from our nonlinearity $f$ in order to avoid another critical exponent related to the decay property of solutions.
To explain these issues, we start with a brief review of the more familiar power case.


\subsection{The energy critical {\sf NLKG}}
In any space dimension $d\geq 1$, the monomial defocusing nonlinear
Klein-Gordon equation reads \EQ{ \label{NLKGp}
 \ddot u - \Delta u + u + |u|^{p-1}u = 0,\quad
 u:\R^{1+d}\to\R.}
The mass term $u$ is irrelevant for local time behavior, so that we can ignore it in the well-posedness issue, but it has essential impacts on the long time behavior, so we must  distinguish it from the massless wave in the scattering theory.

The global solvability in the energy space of (\ref{NLKGp}) has a
long history. The Sobolev critical power $p$ appears when $d\ge 3$, namely
$p^*:=\frac{d+2}{d-2}=\frac{2d}{d-2}-1$, and there are mainly three cases.

In the subcritical case $(p<p^*)$, Ginibre and Velo finally proved in \cite{GV1}
the global well-posedness in the energy space, extending several preceding works which had limitations in the range of power and/or the solution space.

The critical case $(p=p^*)$ is much more delicate.
The global existence of smooth solutions was first proved by Struwe \cite{Stru} in the radially symmetric case, then by Grillakis \cite{Gr1, Gr2} without the symmetry assumption.
For the energy space, Ginibre, Soffer and Velo \cite{GSV} proved the global well-posedness in the radial case, and then Shatah-Struwe \cite{SS} in the general case.
We note that the uniqueness in the energy space is not yet fully settled (see \cite{P} for the case $d\ge 4$ and \cite{Stru2,MP} for partial results in $d\ge 3$.)

The supercritical case $(p>p^*)$ is even much harder and the question remains essentially open,
except for the existence of global weak solutions \cite{Str}, and
some negative results about non-smoothness of the solution map \cite{Leb1}, and loss of regularity \cite{Leb2}.

Concerning the scattering theory of (\ref{NLKGp}) another critical
value of $p$ appears, namely $p_*:=1+\frac{4}{d}$. It is linked to the space-time property of the linear solutions. More precisely, the optimal space-time integrability of free Klein-Gordon solutions in the energy space is given by the Strichartz estimate
\EQ{
 \|u\|_{L^q(\R^{1+d})} \le C E_0(u)^{1/2}, \quad 2 + \frac{4}{d} \le q \le 2+\frac{6}{d-2},}
and $p_*$ is determined by the relation $p_*q/(q-1)=q$ for the
smallest $q=2+4/d$. The scattering between these two powers
$p_*<p<p^*$ was solved in \cite{Brenner,GV2} for $d\ge 3$, and it
was later extended to $d\le 2$ in \cite{2Dsubcrit}, by generalizing
the Morawetz estimate to lower dimensions. The scattering for the
Sobolev critical case $p=p^*$ was solved in \cite{3Dcrit}, but the
lower critical case $p=p_*$ still remains open, which is the reason
that we removed the cubic term from our nonlinearity $f$. We remark
that the scattering for the massless wave equation is available in
the energy space only if the nonlinearity is dominated by the
Sobolev critical power $p^*$ (see \cite{GV3} for the
sub-super-critical case and \cite{BS,BG} for the critical case).

For $d\le 2$, we have no upper bound for the power nonlinearity \eqref{NLKGp}.
But if we consider more general nonlinearity, we still need some growth condition in the two dimensional case, because of the failure of the Sobolev embedding $H^1(\R^2)\subset L^\I(\R^2)$.
Then the exponential nonlinearity \eqref{NLKG} naturally arises in the connection to the Trudinger-Moser inequality.
One should note, however, the size of solution becomes more crucial for the exponential nonlinearity than the power case.
Indeed, solutions of smaller size can be regarded heuristically as giving smaller power in the exponential (it is true if the size is measured in $L^\I$).

Thus the theory for the exponential nonlinearity \eqref{NLKG} was worked out first by Nakamura-Ozawa in \cite{NO1,NO2} for sufficiently small data in the energy space, based on the Trudinger-Moser and the Strichartz estimates, establishing the global existence as well as the scattering of such small solutions.
Later on, the size of the initial data for which one has global existence was quantified, first in \cite{A} for radially symmetric initial data $(0,u_1)$, and then in \cite{2Dglobal,2Dinst} for general data in the energy space.

Furthermore, \cite{2Dglobal} and \cite{2Dinst} established the following trichotomy in the dynamic. This trichotomy is similar to the power case for $d\ge 3$, but the difference is that the threshold is given by the energy size, instead of the power.
More precisely,
\begin{defi}
The Cauchy problem \eqref{NLKG} is said to be {\it subcritical} if
$E(u,0)<1$, {\it critical} if $E(u,0)=1$ and {\it supercitical} if
$E(u,0)>1$.
\end{defi}

Indeed, one can construct a unique local solution if $\|\nabla u_0\|_{L^2}<1$, and the time of existence depends only on $\eta := 1 -\|\nabla u_0\|_{L^2}^2 $.
Actually, it can be constructed even without any size restriction \cite{IMM5}, but then the existence time depends fully on the initial data, not only in terms of the norm.

In the subcritical case, the conservation of energy gives a priori lower bound on $\eta$,
hence the maximal local solutions are global in the subcritical case.
In the critical case, the situation is much more delicate, and arguments based on
the non-concentration of the energy were investigated to extend the
local solutions.

\begin{thm}[Global well-posedness \cite{2Dglobal}]
\label{GWP} Assume that $E(u,0)\leq 1$. Then \eqref{NLKG} has a
unique\footnote{Note that the uniqueness is in the energy space.
This is different from the Sobolev critical case in
 three dimensions, where only partial results are
 available \cite{Stru2,MP}.} global solution $u\in {\mathcal C}({\mathbb R},
H^1)\cap\; {\mathcal C}^1({\mathbb R}, L^2)$. Moreover, $u\in L^4_{loc}({\mathbb R}, {\mathcal C}^{1/4})$.
\end{thm}

The local well-posedness proof of Theorem \ref{GWP} is based on a fixed point argument. Thus, the solution map (data gives solution) is uniformly continuous.
However, in the super-critical case, it is not \cite{2Dinst}.

We should emphasize that the smallness condition in the works by Nakamura-Ozawa is essentially below the threshold (at most $1/2$), even for the global existence part.
This is because the Trudinger-Moser controls only the $L^1$ norm of the critical nonlinearity, which is insufficient for the wave equation in the energy space.
Hence the smallness was used for them effectively to improve the exponent to $L^2$.
This gap was filled out in \cite{2Dglobal} by using a logarithmic inequality.

Moreover, for the scattering problem, the smallness condition was used as the only source to ensure contractiveness for the fixed point argument globally in time.
It is very unlikely for such an argument to hold in the full range of the subcritical regime, no matter how we improve the estimates.
Hence we need more elaborate arguments to use global dispersion of the nonlinear solutions as another source of smallness, just as in the scattering results in the power case without size restriction.


\subsection{The energy critical {\sf NLS}}
There are almost parallel stories for the nonlinear Schr\"odinger equations.
Recall the monomial defocusing
 semilinear Schr\"odinger equation in space dimension
 $d\geq 1$
\begin{equation}
\label{NLSp} i\, \dot u + \Delta u =  |u|^{p-1} u,\quad u: \R^{1+d}
\longmapsto \C,
\end{equation}
which has the same critical exponents $p^*=\frac{d+2}{d-2}$ (only for $d\ge 3$) and $p_*=1+\frac{4}{d}$.

For the {\it{energy subcritical}} case ($p < p^*$), an iteration
of the local-in-time well-posedness result using the {\it{a priori}}
upper bound on $\| u(t) \|_{H^1}$ implied by the conservation laws
establishes global well-posedness for \eqref{NLSp} in $H^1$. Those
solutions scatter when $p>p_*$ \cite{GV4,2Dsubcrit}.

The {\it{energy critical}} case ($p=p^*$) was actually harder than the Klein-Gordon (wave), for which the finite propagation property was crucial to exclude possible concentration of energy, whereas there is no upper bound on the propagation speed for the Schr\"odinger.
Nevertheless, based on new ideas such as induction on the energy size and frequency split propagation estimates, Bourgain \cite{Bourgain1} proved
the global well-posedness and the scattering for radially symmetric data,
and it was extended to the general case by \cite{CKSTT} using a new interaction Morawetz  inequality.

For the exponential nonlinearity in two spatial dimensions, small data global well-posedness together with the scattering was worked out by Nakamura-Ozawa in \cite{NO3}.
Later on, the size
 of the initial data for which one has local existence was quantified for \eqref{NLS} in \cite{CIMM}, and a notion of criticality was proposed:

\begin{defi}
 The Cauchy problem \eqref{NLS} is said to be {\it subcritical} if
$H(u_0)<1$,
 {\it critical} if $H(u_0)=1$ and {\it supercritical} if
$H(u_0)>1$.
\end{defi}

 Indeed, one can construct a unique local solution if $\|\nabla
 u_0\|_{L^2}<1$, and the time of existence depends only on $\eta:=1-\|\na u_0\|_{L^2}$
and $\| u_0\|_{ L^2} $. Hence the maximal local solutions
are indeed global in the subcritical case. The critical case is more
delicate due to the possible concentration of the Hamiltonian. The
following result is proved in \cite{CIMM}.
\begin{thm}[Global well-posedness \cite{CIMM}]
\label{GWP-NLS}
 Assume that $H(u_0)\le 1$, then the problem \eqref{NLS}
has a unique global solution $u$ in the class
$$
{\mathcal C}(\R, H^1(\R^2)).
$$
Moreover, $u\in L^4_{loc}(\R,\;{\mathcal C}^{1/2}(\R^2))$ and
satisfies the conservation laws \eqref{mass} and \eqref{hamil}.
\end{thm}

Our first goal in this paper is to show that every global solution
in the sub- and critical cases to \eqref{NLKG} approaches solutions
to the free Klein-Gordon equation \EQ{\label{LKG}
 \ddot v - \Delta v + v =0,}
in the energy space
$(u(t),\dot u(t))\in H^1(\R^2)\times L^2(\R^2)$ as $t\to\pm\I$.
 This can be done by showing global uniform Strichartz estimtes.
The main result reads.

\begin{thm} \label{nlkg}
\label{Main} For any solution $u$ of \eqref{NLKG} satisfying
$E(u,0)\leq 1$, we have $u\in L^4(\R, {\mathcal C}^{1/4})$ and there
exist unique free Klein-Gordon solutions $u_{\pm}$ such that
$$
E_0(u-u_{\pm},t)\to 0\qquad (t\to\pm \infty).
$$
Moreover, the maps
$$
(u(0),\dot{u}(0))\longmapsto(u_{\pm}(0),\dot{u_{\pm}}(0))
$$
are homeomorphisms between the unit balls in the nonlinear energy
space and the free energy space, namely from $\{(\varphi,\psi)\in
H^1\times L^2 \ ;\ \|\varphi\|_{H^1}^2+\|\psi\|_{L^2}^2+2\|F(\varphi)\|_{L^1}\leq 1\}$
onto $\{(\varphi,\psi)\in H^1\times L^2\ ;\ \|\varphi\|_{H^1}^2+\|\psi\|_{L^2}^2\leq 1\}$.
\end{thm}

The peculiarity of this equation is that the Strichartz
norms give time-local control of the nonlinearity in the
case \eqref{scrit}, but which is not uniform as $\|\na u\|_{L^2}\to 1$,
and so in the critical case $E(u)=1$, it is not {\it a priori}
uniform in time globally, even for a fixed solution.
In this respect, our critical case $E(u)=1$ appears harder than the Sobolev
critical case in higher dimensions \cite{3Dcrit},
where the Strichartz norms give uniform control of the
nonlinearity on time intervals where they are small.

To have the local uniform estimates of $f(u)$, we choose a norm
which takes into account the two behaviors at zero and at infinity
of the nonlinearity. Those estimates can be proved if the
concentration radius of the $H^1$ norm is not too small. (See
section 3 for the precise definition). Having these estimates in
hand, the proof proceeds in the subcritical case $E(u)<1$ almost
verbatim as in \cite{3Dcrit}. However, in the critical case we need
to keep track of the energy distribution and its propagation much
more carefully, because of the non-uniform nature of those estimates.

The second goal in this paper is to show that every global solution
of \eqref{NLS} with $H(u)\leq 1$ approaches solutions to the
associated free equation\EQ{\label{SL}
 i\dot v + \Delta v=0,}
in the energy space $H^1$ as $t\to\pm\infty$. Unfortunately, we have
not succeeded to handle the critical case $H(u)=1$ and we have to restrict
ourselves to the subcritical one.
The reason is that to trace the concentration radius, the finite propagation of energy is essential in our argument for NLKG, which is not available for NLS.
The main ingredient for the subcritical NLS is a new interaction Morawetz
estimate, proved independently by Colliander et al. and Planchon-Vega
\cite{CGT, PV}. This estimate give a priori global bound of $u$ in
$L^4_t (L^8_x)$. Hence, by complex interpolation we deduce that
some of the Strichartz norms used in the nonlinear estimate
go to zero for large time and the scattering in the subcritical
case follows. More precisely, we have
\begin{thm}
 \label{Main-NLS}
For any global solution $u$ of \eqref{NLS} in $H^1$ satisfying
$H(u)<1$, we have $u\in L^4(\R, {\mathcal C}^{1/2})$ and
there exist unique free solutions $u_{\pm}$ of \eqref{SL} such that
$$
\|(u-u_{\pm})(t)\|_{H^1}\to 0\qquad (t\to\pm\infty).
$$
Moreover, the maps
$$
u(0)\longmapsto u_{\pm}(0)
$$
are homeomorphisms between the unit balls in the nonlinear energy
space and the free energy space, namely from $\{\varphi\in H^1\ ; \ H(\varphi)< 1\}$ onto $\{\varphi\in H^1\ ;\  \|\nabla\varphi\|_{L^2}<1\}$.

 \end{thm}
Thus the proof in the subcritical case is much simpler for NLS than NLKG, given the a priori estimate due to \cite{CGT,PV}.
Unfortunately, this type of Morawetz estimates is so far specialized to the Schr\"odinger, and does not seem easily to apply to the Klein-Gordon equation in any space dimension.

This paper is organized as follows. Section
two is devoted to fix the necessary notation
we use.
For the convenience of the reader, we
recall without proofs some useful and quite ``standard''
lemmas.
In section three, we establish the uniform local
estimates.
In section four, we basically follow the arguments
 in \cite{3Dcrit}, proving a uniform global bound on the
  Strichartz norms of solutions under some control on the energy concentration.
This implies in particular the
  scattering Theorem \ref{Main} in the subcritical case.
  Section six is devoted to the critical case.
  First we treat the two extreme (and somehow easier)
  cases when the solution either completely disperse
  or does not disperse at all at the time infinity.
The wider
  case is then when the solution repeats dispersion
  and reconcentration. This necessitates a careful
  study case by case. Section seven is devoted to
  show the optimality of the condition on the local
  $H^1$ norm. The last section treats the case of
  NLS equation. The proof of the scattering in the
   sub-critical case is mainly based on a recent
   {\it a priori} bound on the solutions established
    independently by Colliander et al. \cite{CGT}
    and Planchon-Vega \cite{PV}.


\section{Notation and useful Lemmas}
In this section, we introduce some notation and recall several lemmas we use to prove the main result.

First, let $B^\sigma_{p,q}$ be the inhomogeneous Besov space. In particular, recall that the H\"older space ${\mathcal C}^\sigma=B^\sigma_{\infty,\infty}$, for any non-integer $\sigma>0$. Also, we introduce the following function spaces for some specific Strichartz norms.
\newcommand{\HL}{H_{loc}}
Define the following local $H^1$ norm:
\EQ{
 \|\fy\|_{H^1[R]} := \sup_{c\in\R^2} \int_{|x-c|\le R} |\na\fy(x)|^2 + |\fy(x)|^2 dx,}
and set
\EQ{
 &H:=L^\I_t(H^1), \quad \HL:=L^\I_t(H^1[6]), \quad B:=L^\I_t(B^{-1/4}_{\I,\I}),\\
 &X:= L^{8}(L^{16}), \quad K:=L^4_t(B^{1/4}_{\I,2}\cap B^{1/2}_{4,2}),\\
 &Y_1:=L^{1/\de}(I;L^{2/\de}), \quad Y_2:=L^{4/(1-\de)}(I;C^{1/4-\de}\cap L^8),}
and $Y:=Y_1\cap Y_2$.

Second, define the linear and nonlinear energy densities as
\begin{eqnarray}
\label{LE}
e_L(u,t):=|\dot u|^2+|\nabla u|^2+|u|^2\quad\hbox{    and,   }\quad e_N(u,t):=e_L(u,t)+2F(u),
\end{eqnarray}
respectively. Recall that $F$ is given by \eqref{F}. In addition,
define ${\mathcal G}(u)=uf(u)-2F(u)$, $r=|x|$,
$\omega=\frac{x}{|x|}$, $u_r=\omega\cdot\nabla u$, and \EQ{
\label{Q}
 t^2Q(u):=&(t\dot u+ru_r+u)^2+(r\dot u+tu_r)^2\\
 &+(t^2+r^2)(|\na u|^2-|u_r|^2+u^2).}
Then we have the inversional identity (see \cite{Strauss,Gr2} or more specifically \cite[(7.4)]{3Dcrit}):
\begin{lem}
\label{II}
For any constant $c>0$ and for any solution $u$ of \eqref{NLKG} with $E(u)\le 1$, we have
\EQ{
\left[\int_{r<ct}t^2Q(u)+2(t^2+r^2)F(u)dx\right]_S^T
 =&\int_{cS<r<cT}P_c(u)dx\\
  &+\int_{S}^T\int_{r<ct}4t(u^2-H(u))dxdt,}
where $H(u)=\frac{{\mathcal G}(u)}{2}-2F(u)$, and $P_c(u)$  satisfies
$$
|P_c(u(r/c,x))|\lec \left[e_N(u)+\frac{u^2}{r^2}\right](r/c,x).
$$
\end{lem}
The formal proof does not require the condition $E(u)\le 1$. We use it only to extend the estimate from smooth solutions to the energy class by the well-posedness.
We also need the generalized Morawetz estimate.
\begin{lem}[Morawetz estimate \cite{2Dsubcrit}, Lemma 5.1]
There exists a positive constant $C$ such that for any solution $u$
of \eqref{NLKG} with energy $E(u)\leq 1$, one has
\begin{eqnarray}
\label{ME}
\int_{{\mathbb R}^{2+1}} \frac{|x\dot u + t\na u|^2 + |x\times\na u|^2 + (1+t^2){\mathcal G}(u)}{1+|t|^3+|x|^3}dxdt\leq C.
\end{eqnarray}
\end{lem}
To establish uniform estimates on the nonlinearity, we should make sure that energy does not concentrate. This can be quantified through the following notion
\begin{defi}
For any $\fy\in H$ such that $\partial_t\fy\in L^\infty_t(L^2)$, and any $A>0$, we denote the concentration radius of $A$ amount of the nonlinear energy by
\EQ{
 R_A[\fy](t) := \inf\{r>0 \mid \exists c\in\R^2,\ \int_{|x-c|<r}
 e_{N}(\fy,t)\,dx > A\}.}
\end{defi}
Observe that lower bounds on the concentration radius yield upper bound on the the local $H^1$ norm. More precisely
\EQ{ \label{equiv}
 R_A[\fy] (t)\ge R \quad    \Rightarrow \quad  \|\fy\|_{L^\infty_t(H^1[R])} \le A.}

Finally, recall the sharp Trudinger-Moser inequality on $\R^2$ \cite{Ruf}. It is the limit case of the Sobolev embedding. For any $\mu>0$ we have
\EQ{ \label{Trudinger-Moser}
 \sup_{\|\fy\|_{H_\mu} \le 1} \int (e^{4\pi|\fy|^2}-1) dx <\I,}
where $H_\mu$ is defined by the norm $\|u\|_{H_\mu}^2:=\|\nabla u\|_{L^2}^2+\mu^2\|u\|_{L^2}^2$.
We can change $\mu>0$ just by scaling $\fy(x)\mapsto\fy(x/\mu)$.

It is known that the $H^1({\mathbb R}^2)$ functions are not generally in $L^\I$.
The following lemma shows that we can estimate the $L^\infty$ norm
by a stronger norm but with a weaker growth (namely logarithmic).
\begin{lem}[Logarithmic inequality \cite{DlogSob}, Theorem 1.3]
Let $0<\alpha<1$. For any real number $\lambda>\frac{1}{2\pi\alpha}$, a constant $C_\lambda$ exists such that for any function $\varphi\in H^1_0\cap{\dot{\mathcal C}}^\alpha(|x|<1)$, one has
\begin{eqnarray}
\label{LS}
\|\varphi\|_{L^\infty}^2\leq \lambda\|\nabla\varphi\|_{L^2}^2 \log\left(C_\lambda+\frac{
\|\varphi\|_{{ \dot{\mathcal C}}^\alpha }}{ \|\nabla\varphi\|_{L^2} }\right).
\end{eqnarray}
\end{lem}

We also recall the whole space version of the above inequality.
\begin{lem}[\cite{DlogSob}, Theorem 1.3]
\label{Hmu}
 Let $0<\alpha<1$.  For any $\lambda>\frac{1}{2\pi\alpha}$ and
any $0<\mu\leq1$, a constant $C_{\lambda}>0$ exists such that, for
any function $u\in H^1(\R^2)\cap{\mathcal C}^\alpha(\R^2)$
\begin{equation}
\label{H-mu} \|u\|^2_{L^\infty}\leq
\lambda\|u\|_{H_\mu}^2\log\left(C_{\lambda} +
\frac{8^\alpha\mu^{-\alpha}\|u\|_{{\mathcal
C}^{\alpha}}}{\|u\|_{H_\mu}}\,\,\,\right).
\end{equation}
\end{lem}


\section{Local uniform estimate on the nonlinearity}
In this section, we estimate the nonlinearity in terms of the Strichartz norms in the critical case, such that the constant in the inequality does not depend on the interval as long as the concentration radius is not too small.

\begin{lem} \label{crit-nlst}
For any $A\in(0,1)$, there exists $\de\in(0,1/8)$, and a continuous increasing function $C_A:[0,\I)\to[0,\I)$, such that for any $\fy_0,\fy_1\in H^1(\R^2)$ satisfying
\EQ{ \label{H6}
  \|\fy_j\|_{H^1[6]}\le A,\quad(j=0,1)}
we have, putting $\fy=(\fy_0,\fy_1)$
\EQ{
\label{NLE0}
 \|f(\fy_1)-f(\fy_0)\|_{L^2} \le C_A(\|\fy\|_{H^1}) \|\fy\|_{C^{1/4-\de}\cap L^8}^{4}\|\fy_1-\fy_0\|_{L^{2/\de}}.}
\end{lem}
Then by integration and the H\"older inequality in $t$, we obtain
\begin{cor} \label{nst-crt}
Let $A\in(0,1)$, $\de\in(0,1/8)$ and $C_A$ be as above. Let $I\subset\R$ be any measurable subset and assume that $u_0,u_1\in L^\I(I;H^1(\R^2))$ satisfy
\EQ{
 \sup_{t\in I} \|u_j(t)\|_{H^1[6]} \le A\quad (j=0,1).}
Then we have, putting $u=(u_0,u_1)$,
\EQ{
\label{NLE}
 \|f(u_1)-f(u_0)\|_{L^1(I;L^2)}
 \le C_A(\|u\|_{L^\I(I;H^1)}) 
  \|u_1-u_0\|_{Y_1}\|u\|_{Y_2}^4.}
In particular, if
$$\inf_{t\in I}R_A[u_j](t)\geq6,$$
then \eqref{NLE} holds.
\end{cor}
\begin{rem}
\begin{enumerate}
\item The estimate on $f(u)$ itself follows by putting $u_0=0$ in the above.
\item By the Strichartz estimate for the Klein-Gordon equation, we have
\EQ{ \label{Stz}
 \|u\|_{L^4((0,T);B^{1/4}_{\I,2}\cap B^{1/2}_{4,2})} \lesssim \|\ddot u -\De u + u\|_{L^1((0,T);L^2)}
 + \|u(0)\|_{H^1} + \|\dot u(0)\|_{L^2}.}
\item By the complex interpolation and the embedding between the Besov spaces, we have
\EQ{ \label{intp1}
 &\|u\|_{Y_1}=\|u\|_{L^{1/\de}_t(L^{2/\de})} \le C_\de \|u\|_{L^\I_t(H^1)}^{1-8\de} \|u\|_{L^8_t(L^{16})}^{8\de}=C_{\de}\|u\|_H^{1-8\de}\|u\|_X^{8\de},\\
 &\|u\|_{Y_2}=\|u\|_{L^{4/(1-\de)}(C^{1/4-\de}\cap L^8)}
  \lec  \|u\|_{L^\I(H^1)}^\de \|u\|_{L^4(B^{1/4}_{\I,2}\cap B^{3/8}_{8,2})}^{1-\de},\\
 &\|u\|_X = \|u\|_{L^8_t(L^{16})} \lec \|u\|_{L^\I(B^{-1/4}_{\I,\I})}^{1/2} \|u\|_{L^4_t(B^{3/8}_{8,2})}^{1/2}.}
Also note that $B^{3/8}_{8,2}=[B^{1/4}_{\I,2},B^{1/2}_{4,2}]_{1/2}$.
\end{enumerate}
\end{rem}

The above lemma is reduced to that on a disk of radius $6$ by the following cubic decomposition. We define a radial cut-off function $\chi\in C^{1,1}(\R^2)$ by
\EQ{
 \chi(x) = \begin{cases}
 1, &(|x|\le 2)\\
 1-(|x|-2)^2/8, &(2 \le |x| \le 4)\\
 (|x|-6)^2/8, &(4 \le |x| \le 6)\\
 0, &(|x|\ge 6).
 \end{cases}}
Then we have $0\le\chi\le 1$, $\supp\chi=\{|x|\le 6\}$,
$|\na\chi|\le|\chi'|\le 1/2$, $|\De\chi|\le 3/8$.
For any $\fy:\R^2\to\R$ and $(j,k)\in\Z^2$, we define
\EQ{ \label{cubic decop}
 \fy^{j,k}(x) := \chi(x-2\sqrt{2}(j,k))\fy(x).}
Then for any $x\in\R^2$, there exists $(j,k)\in\Z^2$ such that for any function $\fy:\R^2\to\R$,
\EQ{
 \fy(x) = \fy^{j,k}(x).}
In particular, we have for any function $g:\R\to\R$,
\EQ{ \label{nonlin decop}
 |g(\fy(x))| \le \sum_{j,k\in\Z} |g(\fy^{j,k}(x))|.}
By the support property, we have also
\EQ{ \label{L2 decop}
 \sum_{j,k\in\Z} \|\fy^{j,k}\|_{L^2}^2 \lec \|\fy\|_{L^2}^2,\quad
 \sum_{j,k\in\Z} \|\na\fy^{j,k}\|_{L^2}^2 \lec \|\fy\|_{H^1}^2.}
Moreover, by the bounds on $\chi$, $\na\chi$ and $\De\chi$, we have for any $p\in[1,\I]$ and any $\th\in(0,1)$,
\EQ{ \label{local bounds}
 \pt \|\fy^{j,k}\|_{L^p} \le \|\fy\|_{L^p},\quad \|\fy^{j,k}\|_{C^\th} \le \|\fy\|_{C^\th},
 \pr \|\na \fy^{j,k}\|_{L^2}^2 + \|\fy^{j,k}\|_{L^2}^2/2 \le \|\fy\|_{H^1[6]}^2.}
 To derive the above $\dot H^1$ bound, we use in addition the formula
\EQ{ \label{H1 cutoff}
 \|\na(\psi\fy)\|_{L^2}^2 = \int (\psi^2|\na\fy|^2 - \psi\De\psi|\fy|^2) dx,}
with $\psi:=\chi(x-2\sqrt{2}(j,k))$.

\begin{proof}[Proof of Lemma \ref{crit-nlst}]
Without loss of generality, we may assume that $A\in(1/2,1)$.
We can choose $\de\in(0,1/8)$ and $\la>2/\pi$, depending only on $A$, such that
\EQ{
 2\pi\la(1/4-\de) > 1 =  2\pi\la (1/4+\de/2) A^2.}
For instance, we can choose $\delta=\frac{1-A^2}{8}$. 
By the mean value theorem we have
\EQ{
 &f(\fy_1)-f(\fy_0) = \int_0^1 f'(\fy_\th) d\th\times  \fy',\\
 &\fy_\th:=\th\fy_1+(1-\th)\fy_0,\quad \fy':=\fy_1-\fy_0,}
so it suffices to bound $\|f'(\fy_\th)\fy'\|_{L^2}$ uniformly for $\th\in[0,1]$. We consider the two cases $\|\fy_\th\|_{L^\I}\le A$ and $\|\fy_\th\|_{L^\I}>A$ separately. In the former case, we have
\EQ{
 \|\fy_\th\|_{L^\I} \le A <1,}
hence
\EQ{
 |f'(\fy_\th(x))| \lec |\fy_\th(x)|^4}
pointwise in $x$, so that we can bound it by the H\"older and the Sobolev inequalities
\EQ{
 \|f'(\fy_\th)\fy'\|_{L^2}
 &\lec \|\fy_\th\|_{L^8}^{4(1-\de)} \|\fy'\|_{L^{2/\de}} \|\fy_\th\|_{L^\I}^{4\de}.}
\newcommand{\fjk}{\fy^{j,k}}
In the remaining case $\|\fy_\th\|_{L^\I}>A$, we apply the cubic decomposition to $\fy_\th$:
\EQ{
 |f'(\fy_\th)| \le \sum_{j,k\in\Z^2} |f'(\fjk_\th)|.}
Then \eqref{local bounds} implies
\EQ{
 \|\na\fjk_\th\|_{L^2}^2 + \|\fjk_\th\|_{L^2}^2/2 \le \|\fy_\th\|_{H^1[6]}^2 \le A^2 <1.}
We apply the H\"older inequality together with the general pointwise estimate
\EQ{
 |f'(a)| \lec |a|^2(e^{4\pi|a|^2}-1),}
to the decomposed functions,
\EQ{
 &\|f'(\fjk_\th)\fy'\|_{L^2}
 \le  \|e^{4\pi|\fjk_\th|^2}-1\|_{L^1}^{1/2-\de} \|e^{4\pi|\fjk_\th|^2}-1\|_{L^\I}^{1/2+\de} \|\fjk_\th\|^2_{L^{4/\de}} \|\fy'\|_{L^{2/\de}}.}
The first factor on the right is bounded by the sharp Trudinger-Moser inequality \eqref{Trudinger-Moser},
for the third factor we use the Sobolev inequality
\EQ{ \label{Lp bound}
 \|\fjk_\th\|_{L^{4/\de}} \lec \|\fjk_\th\|_{H^1},}
and the second factor is bounded by using the logarithmic inequality \eqref{LS}
\EQ{ \label{exponential estimate}
 \le e^{2\pi(1+2\de)\|\fjk_\th\|_{L^\I}^2}
  &\le (C_\la+\|\fjk_\th\|_{C^{1/4-\de}}/\|\na\fjk_\th\|_{L^2})^{2\pi(1+2\de)\la\|\na\fjk_\th\|_{L^2}^2}\\
 &\le (C_\la+\|\fy_\th\|_{C^{1/4-\de}}/A)^{2\pi(1+2\de)\la A^2}
  \le (4C_\la \|\fy_\th\|_{C^{1/4-\de}})^{4},}
since $\|\fy_\th\|_{C^{1/4-\de}}>A>1/2$ and $C_\la\ge 1$. Thus we obtain
\EQ{
 \|f'(\fjk_\th)\fy'\|_{L^2} \lec \|\fjk_\th\|_{H^1}^{2} \|\fy\|_{C^{1/4-\de}}^{4} \|\fy'\|_{L^{2/\de}}.}
Then we arrive at the desired estimate after the summation over $(j,k)\in\Z^2$ and using \eqref{L2 decop}.
\end{proof}

\subsection{Nonlinear estimate for the critical case}
In the critical case $E(u)=1$, we will need the nonlinear estimate assuming the subcritical local energy only for one of the functions. It is generally impossible, but we can retain it by additional smallness in the critical Besov space for the difference, which is available when the difference corresponds to a highly localized concentration of energy.
\begin{lem} \label{critical estimate}
For any $A\in(0,1)$, there exist $\de(A),\ups(A)\in(0,1/8)$, and a continuous increasing function $C_A:[0,\I)\to[0,\I)$, such that for any $\fy_0,\fy_1\in H^1(\R^2)$ satisfying
\EQ{
  \|\fy_0\|_{H^1[6]}\le A, \pq \|\fy_1\|_{H^1[6]} \le 1, \pq \|\fy_1-\fy_0\|_{B^0_{\I,2}} \le \ups(A),}
we have the estimate \eqref{NLE0}.
\end{lem}
\begin{cor} \label{crit nonlin}
Let $A\in(0,1)$, $\de(A),\ups(A)\in(0,1/8)$, and $C_A$ be as above. Let $I\subset\R$ be any measurable subset and assume that $u_0,u_1\in L^\I(I;H^1(\R^2))$ satisfy for all $t\in I$,
\EQ{
 \|u_0(t)\|_{H^1[6]} \le A, \pq \|u_1(t)\|_{H^1[6]}\le 1, \pq \|u_1(t)-u_0(t)\|_{B^0_{\I,2}} \le \ups(A).}
Then we have the estimate \eqref{NLE}.
\end{cor}

It suffices to prove the lemma, which is almost the same as before.
\begin{proof}
The only difference appears in the use of the logarithmic inequality \eqref{exponential estimate}, where the assumption $A<1$ was essential. Now we choose
$\de\in(0,1/8)$ and $\la>\frac{2}{\pi}$ such that
\EQ{
 \pt 2\pi\la(1/4-\de) > 1 = 2\pi\la(1/4+\de/2)(1+A)^2/4, \pr 1-\e/4 = 2\pi\la(1/4+\de/2)A^2,}
for some $\e\in(0,1)$.

In order to use the new assumption, we expand the exponential in \eqref{exponential estimate}. For simplicity, we omit the superscript $j,k$ for a while.
\EQ{ \label{exp expand}
 e^{2\pi(1+2\de)\|\fy_\th\|_{L^\I}^2} \le e^{2\pi(1+2\de)\|\fy_0\|_{L^\I}^2} e^{2\pi(1+2\de)2\th\|\fy_{\th/2}\|_{L^\I}\|\fy'\|_{L^\I}}.}
For the first term we apply the same estimate as in \eqref{exponential estimate}, thus it is bounded by
\EQ{
 (4C_\la\|\fy_0\|_{C^{1/4-\de}})^{4-\e}.}
Here and after, the constant $C_\la\ge 1$ is determined by $\la$, but may change from line to line.
For the second term, we use
\EQ{
 \pt\|\fy_{\th/2}\|_{L^\I}^2 \le \la\|\na\fy_{\th/2}\|_{L^2}^2 \log(C_\la + \|\fy_{\th/2}\|_{C^{1/4-\de}}/\|\na\fy_{\th/2}\|_{L^2})
  \pr\qquad\quad \le \la \log(C_\la + \|\fy_{\th/2}\|_{C^{1/4-\de}}),
 \pr\|\fy'\|_{L^\I}^2 \lec \|\fy'\|_{B^0_{\I,2}}^2 \log(C+\|\fy'\|_{C^{1/4-\de}}/\|\fy'\|_{B^0_{\I,2}})
  \pr\qquad\quad \lec \|\fy'\|_{B^0_{\I,2}} \log(C+\|\fy'\|_{C^{1/4-\de}}),}
for some $C\ge 1$ and $C_\la\ge 1$, where we used the monotonicity of $x\log(C+B/x)$.
For the latter inequality we refer to \cite[Theorem 2.1]{KOT}.
Using the convexity of $\log$
\EQ{
 \log(C+A)\log(C+B) \le [\log(C+(A+B)/2)]^2,}
we obtain
\EQ{
 \|\fy_{\th/2}\|_{L^\I}\|\fy'\|_{L^\I} \le C_\la\|\fy'\|_{B^0_{\I,2}}^{1/2}\log\{C_\la+(\|\fy_{\th/2}\|_{C^{1/4-\de}}+\|\fy'\|_{C^{1/4-\de}})/2\}.}
Hence by choosing $\ups\ll\e^2$, the second term in \eqref{exp expand} is bounded by
\EQ{
 (C_\la + \|\fy\|_{C^{1/4-\de}})^\e \lec \|\fy\|_{C^{1/4-\de}}^\e,}
where we used that $\|\fy\|_{C^{1/4-\de}}\ge\|\fy_\th\|_{C^{1/4-\de}}>A>1/2$.
The remaining argument is just the same as before.
\end{proof}

\subsection{$L^\I-L^2$ estimate for localized free Klein-Gordon}
The smallness in the critical Besov space in the above nonlinear estimate will be provided by using a sharp decay estimate given below for the free Klein-Gordon equation.
The decay factor $(t/R)^{(1-d)/2}$ is easily obtained by using the usual decay estimates;
The nontrivial thing is that we do not lose any regularity compared with the Sobolev, nor any decay compared with the $L^\I-L^1$ decay for the wave equation. Here we denote $\LR{\x}:=\sqrt{1+|\x|^2}$ and $\LR{\na}:=\sqrt{1-\De}$.
\begin{lem} \label{LIL2 decay}
Let $d\geq 1$, $\chi\in\S(\R^d)$ and $\al\in[0,d/2]$. Then we have for $R\gec 1$ and $t\in\R$,
\EQ{
 \|e^{it\LR{\na}}\LR{\na}^{-\al}\chi(x/R)f\|_{B^0_{\I,2}} \lec \LR{t/R}^{(1-d)/2}\|f\|_{H^{d/2-\al}}.}
\end{lem}

\begin{proof}
First note that the case $|t|<R$ is trivial by the Sobolev embedding $H^{d/2-\al}\subset B^{-\al}_{\I,2}$.
Hence we assume that $t\gg R>1$.

Next we dispose of the lower frequency. Denote by $P_{\le 1}$ a fixed smooth cut-off in the Fourier space onto $|\x|\le 1$.
We define $P_{\le R}$ for any $R>0$ by rescaling $P_{\le 1}$, and also $P_{>R}:=1-P_{\le R}$. Denote $\chi_R(x)=\chi(x/R)$.
By using the $L^{2d}$ decay for the free Klein-Gordon, and the Sobolev embedding $B^{1/2}_{2d,2}\subset B^0_{\I,2}$,
we have for some large $N\in\N$,
\EQ{
 \|P_{\le 1}e^{it\LR{\na}}\LR{\na}^{-\al}\chi_Rf\|_{B^0_{\I,2}}
 \pt\lec \|e^{it\LR{\na}}\chi_Rf\|_{B^{-N}_{2d,2}}
 \pr\lec t^{(1-d)/2}\|\chi_Rf\|_{L^{2d/(2d-1)}} \lec (t/R)^{(1-d)/2}\|f\|_{L^2}.}
Similarly, we have
\EQ{
 \|e^{it\LR{\na}}\LR{\na}^{-\al}\chi_RP_{\le 1}f\|_{B^0_{\I,2}}
 \pt\lec t^{(1-d)/2}\|\chi_RP_{\le 1}f\|_{B^N_{2d/(2d-1),2}}
 \pr\lec t^{(1-d)/2}\|f\|_{W^{N+1,2d/(2d-1)}} \lec (t/R)^{(1-d)/2}\|f\|_{L^2}.}

Hence it suffices to show
\EQ{
 \|P_{>1}e^{it\LR{\na}}\chi_RP_{>1}f\|_{\dot B^{-\al}_{\I,2}} \lec (t/R)^{(1-d)/2}\|f\|_{\dot H^{d/2-\al}},}
for $t\gg R>1$.
Taking advantage of the homogeneity, we can
rescale by $(t,x)\mapsto(Rt,Rx)$, then the above
is reduced to
\EQ{
 \|P_{>R}e^{it\LR{\na}_R}\chi P_{>R}f\|_{\dot B^{-\al}_{\I,2}} \lec t^{(1-d)/2}\|f\|_{\dot H^{d/2-\al}},}
for $t\gg 1$ and $R>1$, where we denote $\LR{\na}_R:=\sqrt{R^2-\Delta}$.

We decompose the left hand side by the dyadic decomposition, and consider the dyadic piece $|\x|\sim J>R>1$.
We need to estimate the $L^\I_x$ norm of
\EQ{
 I_J := \iint e^{it\sqrt{|\x|^2+R^2}+ix\x}\fy(\x/J)
 \ti\chi(\x-\y)\ti f(\y)d\y d\x,}
where $\fy(\x)$ is the cut-off for the Littlewood-Paley decomposition, satisfying
\EQ{
 \supp\fy\subset\{1/2<|\x|<2\},\pq \sum_{J\in 2^\Z}\fy(\x/J)=1, \pq(\x\not=0).}

Now we argue by the stationary phase estimate.
Let $\Om=t\sqrt{|\x|^2+R^2}+\x x=:t\LR{\x}_R+\x x$, then $\na\Om=t\x/\LR{\x}_R+x$, $|\na^{1+N}\Om|\lec\LR{\x}^{-N}t$.
For the size of $|\na\Om|$, we take account only of its angular component, i.e.
\EQ{
 \th:=|\x/\LR{\x}\times x/|x||,\pq |\na\Om|\ge t\th.}

Let $g\in C^\I_0(\R^d)$ and $\R^d\ni p\not=0$, and assume that $\p_p\Om\not=0$ on $\supp g$. Then we have by the partial integration,
\EQ{
 \int e^{i\Om}g(\x) d\x = \int e^{i\Om}(i\p_p(\p_p\Om)^{-1})^N g(\x) d\x,}
for any $N\in\N$, where both $\p_p$ and $(\p_p\Om)^{-1}$ in the parenthesis should be regraded as operators.
Thus the $N$-th power of the operator is expanded into a linear combination of terms like
\EQ{
 e^{i\Om} \frac{\p_p^{\be_1+1}\Om\cdots \p_p^{\be_K+1}\Om}{(\p_p\Om)^\al}\p_p^\ga,\pq \al=N+K,\pq \be_1+\cdots+\be_K+\ga=N,}
and its integral is bounded by
\EQ{
 \int (t\th)^{-N}(J\th)^{-K}J^{K+\ga-N}|\p_p^\ga g| d\x
  \lec \int (t\th)^{-N}(J\th)^{\ga-N}|\p_p^\ga g| d\x.}

For the part $J\gec t$, we can decompose $\fy(\x/J)\ti\chi(\x-\y)$ into the regions $t\th\lec 1$ and $t\th\gec 1$, by a smooth cut-off of the form $h(t\th)$. Then $|\p_p^\ga h(t\th)|\lec (t/J)^\ga \lec 1$, hence by choosing $N=0$ for $t\th\lec 1$ and sufficiently large $N$ for $t\th\gec 1$, we get
\EQ{
 \left|\int e^{i\Om}\fy(\x/J)\ti\chi(\x-\y)d\x\right| \lec
 \int_{|\x|\sim J} \LR{t\th}^{-N}|\hat\chi(\x-\y)| d\x,}
for some $\hat\chi\in\S$ given by derivatives of $\ti\chi$. Applying the Young on the right hand side ($L^2$ for the first term and $L^1$ for the second), we get
\EQ{
 |I_J| \pt\lec \sum_{l\in\N} \Bigl\|\int_{|\x|\sim 2^j}\LR{t\th}^{-N}|\hat\chi(\x-\y)|d\x\Bigr\|_{L^2_\y(|\y|\sim 2^l)} \|\ti f\|_{L^2_\y(|\y|\sim 2^l)}
 \pr\lec (2^j/t)^{(d-1)/2}2^{j/2} \sum_{l\in\N} 2^{-N|j-l|} \|\ti f\|_{L^2_\y(|\y|\sim 2^l)},}
where we put $J=2^j$, and the two $N$'s are unrelated arbitrary large numbers. Hence by applying the Young for the convolution on $\Z$, we get
\EQ{
 \|J^{-\al}I_J\|_{\ell^2_{J\gec t}L^\I_x} \lec t^{(1-d)/2}\|f\|_{\dot H^{d/2-\al}}.}

If $J\ll t$, then we apply the cubic decomposition of size $J\th_0$ to the $\x$ integral with $\th_0=(tJ)^{-1/2}$, and we use the above partial integration in those cubes with $\th>\th_0$. Since the $\p^\ga$ derivative on the cubic cut-off is bounded by $(J\th_0)^{-\ga}$, we thus get
\EQ{
 |\int e^{i\Om}\fy(\x/J)\ti\chi(\x-\y)d\x|
  \lec \int_{|\x|\sim J} \LR{\th/\th_0}^{-N}|\hat\chi(\x-\y)| d\x,}
where $\hat\chi\in\S$ is given by derivatives of $\ti\chi$.
Now we apply the Young to the right hand side, but with different exponents for different variables.
In the $d-1$ dimensions perpendicular to $x$, we apply $L^1$ and $L^2$ to the first and the second funcitons, respectively.
In the remaining $x$ direction, we apply $L^2$ and $L^1$ to them.
Then we get with $J=2^j\ll t$,
\EQ{
 |I_J| \lec \sum_{l\in\Z} (J\th_0)^{d-1}J^{1/2} 2^{-N|j-l|}\|\ti f\|_{L^2(||\y|\sim 2^l)},}
for arbitrary $N\in\N$. By $\th_0=(tJ)^{-1/2}$, the coefficient equals $t^{(1-d)/2}J^{d/2}2^{-N|j-l|}$.
Hence we get the desired bound as the previous case, by the Young inequality on $\Z$.
\end{proof}


\section{Strichartz estimates in the scattering}

In the following lemma, the global energy bound $E(u)\le 1$ is irrelevant. In fact, we have the same conclusion for arbitrary finite $E(u)$, where the estimates depend also on $E(u)$.
\begin{lem} \label{bump}
Let $u$ be a solution of \eqref{NLKG} on an interval $I$ satisfying
$E(u)\le 1$, $\|u\|_{\HL(I)}\le A<1$ and $\|u\|_{X(I)}=\y<\I$.
There exists a constant $\y_0(A)\in(0,1)$ such that if $\y\le \y_0(A)$ then
we have \EQ{
 \|u\|_{K(I)} \lec 1.}
Moreover, there exist $J\subset I$, $R=R(\y)\sim\y^{-3}$ and $c\in\R^2$ such that $|J|\gec\y^8$ and
\EQ{
 J\ni\forall t,\quad \int_{|x-c|<R} |u(t,x)|^p dx \ge C_p \y^{16}}
for any $p\in[1,8]$.
\end{lem}
\begin{proof}
Let $I=(S,T)$ and let $v$ be a free solution with the same initial data as $u$ at
$t=S$. By the Strichartz estimate \eqref{Stz} and the interpolation
inequalities \eqref{intp1}, we have \EQ{
 \|u-v\|_{K\cap X(I)}
 \le C\|f(u)\|_{L^1(I;L^2)}
 &\le C(A)\|u\|_{K(I)}^{4(1-\de)}\|u\|_{X(I)}^{8\de},}
where $\de\in(0,1/8)$ is given in Lemma \ref{crit-nlst} depending on $A$. Hence we have
\EQ{
 \|u\|_{K(I)} \lec 1 + C(A)\y^{4\de}\|u\|_{K(I)}^{4(1-\de)}.}
Since $4(1-\de)>1$ and $\|u\|_{K(I)}\to 0$ as $T\to S+0$, we obtain the desired {\it a priori} bound if $\y\le \y_0(A)$ is sufficiently small.
Then \eqref{intp1} implies
\EQ{
 \|u\|_{B(I)} \gec \|u\|_{X(I)}^2 \|u\|_{K(I)}^{-1} \gec \y^{2}.}
Hence there exists $T\in I$, $c\in\R^2$, and $N\ge 1$ such that
\EQ{
  \y^{2} \lec N^{-1/4}|u_{N}(T,c)|,}
where $u_N$ stands for the dyadic piece of $u$. The Sobolev embedding bounds the RHS by $N^{-1/4}$, and also $\y=\|u\|_{X(I)}\lec |I|^{1/8}$. Hence we have $N+1/|I|\lec \y^{-8}$. Moreover, by the $L^2$ bound on $\dot u$, we have
\EQ{
 \|u_{N}(T)-u_{N}(t)\|_{L^\I} \lec N \|u(T)-u(t)\|_{L^2} \lec N|T-t|,}
and so there exists $J\subset I$ such that $|J|\gec\y^8$ and
\EQ{
 \y^2\lec N^{-1/4}\inf_{t\in J}|u_{N}(t,c)|.}
Writing the frequency localization by using some Schwarz function $\fy$, we have
\EQ{
 N^{1/4}\y^2 \lec |u_{N}(t,c)| &=|(N^2\fy(Nx)*u)(t,c)|\\
 &\lec N^{2/p} \|u(t)\|_{L^p(|x-c|<R)} + N^{1/4}R^{-3/4}\|u(t)\|_{L^2},}
where we used the H\"older inequality. Since $N\lec\y^{-8}$ and $\|u(t)\|_{L^2}\le 1$, if $R\gg\y^{-8/3}$ and $p\le 8$ then we have
\EQ{
 \|u(t)\|_{L^p(|x-c|<R)} \gec N^{1/4-2/p}\y^2 \gec \y^{16/p}.}
\end{proof}

\begin{lem}
\label{long}
Let $u$ be a solution of \eqref{NLKG} on an interval $I=(S,T)$ with $E(u)\le 1$ and $\|u\|_{\HL(I)}\le A<1$. Let $S=\t_0<\t_1<\cdots<\t_N=T$ such that $\|u\|_{X(\t_j,\t_{j+1})}\ge\y>0$ for all $j$. Then we have
\EQ{
 \sum_{j=1}^N \frac{\min(\y,\y_0(A))^{40}}{\t_j-\t_0+\y^{-3}} \lec 1.}
\end{lem}
\begin{proof}
\newcommand{\Le}{\preceq}
\newcommand{\Ge}{\succeq}
By subdividing each interval $(\t_j,\t_{j+1})$ if necessary, we may assume $\y\le\y_0(A)$ without loss of generality.
By Lemma \ref{bump}, there exists $R\sim \y^{-3}$, $t_j\in(\t_{j-1},\t_{j})$ and $c_j\in\R^2$ such that
\EQ{
 \int_{|x-c_j|<R} |u(t_j,x)|^2 dx \gec \y^{16}.}
Define the following relation on the set of all indices
\EQ{
 j\Le k \quad \iff \quad  [ j\le k \text{ and } |c_k-c_j| \le |t_k-t_j| + 2 R.]}
Then we can inductively find a unique $M\subset\{1,\dots,N\}$ such that
\begin{enumerate}
\item Every $k\in M$ is minimal, i.e. there is no $j\Le k$.
\item Each $j$ has a $k\in M$ satisfying $k\Le j$.
\end{enumerate}
By the first property and the finite propagation of energy, we have
\EQ{ \label{numbound}
 1\ge E(u) \ge \sum_{k\in M} \int_{|x-c_k|<R} e(u(t_k,x)) dx \gec \y^{16} (\#M) ,}
where $\#M$ denotes the number of the minimal elements, and by the generalized Morawetz inequality we have for each $k\in M$,
\EQ{
 1\ge E(u) \gec \int \frac{|u|^6}{|t-t_k|+|x-c_k|+1} dxdt
 \gec \sum_{j\Ge k} \frac{\y^{16}\y^8}{|t_j-t_k|+\y^{-3}}.}
Summing this for $k\in M$, and using the second property of $M$, we obtain
\EQ{
 \#M \gec \sum_{k\in M} \sum_{j\Ge k} \frac{\y^{24}}{|t_j-t_k|+\y^{-3}}
 \gec \sum_{j=1}^N \frac{\y^{24}}{\t_j-\t_0+\y^{-3}},}
which is bounded by $\y^{-16}$ due to \eqref{numbound}.
\end{proof}

\begin{lem}[Space time localized energy]
\label{STLE}
There are functions $L:(0,1)^2\to(1,\I)$ and $\nu_0:(0,1)\to(0,1/2)$ with the following property.
Let $T_1<T_2$ and let $u$ be a solution of \eqref{NLKG} on $(T_1,T_2)$ satisfying $E(u)\le 1$,
\EQ{
 &\|u\|_{H_{loc}(T_1,T_2)}\leq A, \quad \|u\|_{X(T_1,T_2)}\le 2\y,\\
 &\int_{|x-c|<R}e_N(u,T_1) dx \ge \nu^2,\quad |T_1-T_2|\ge L(\nu,\ka)R,}
for some $A\in(0,1)$, $\ka\in(0,1)$, $\y\in (0,\y_0(A)]$, $\nu\in(0,\nu_0(A))$, $c\in\R^2$ and $R\ge 1$.
Then there exists a free solution $v$ and $S\in(T_1,T_2)$ such that
\EQ{
 &E(u-v;S) \le E(u)- \nu^2/2,
 \quad E_0(v)\le 2\nu^2,\\
 & \|v\|_{L^\I B^0_{\I,2}(S,\I)} + \|v\|_{X(S,\I)} \le \ka.}
\end{lem}

\begin{proof}
We may set $T_1=0$ and $c=0$ by space-time translation.
First we seek a bump of energy with very thin Strichartz norm, to be the initial data of $v$ at a time much before $S$.
By Lemma \ref{bump}, we have $\|u\|_{K(T_1,T_2)}\lec 1$, and by the finite propagation,
$$
\int_{|x|<t+R}e_N(u;t)dx \ge \nu^2.
$$
at any $t\ge 0$. Let $M,\nu'$ be large and small numbers, which will be determined later depending on $\nu$ and $\ka$.
Let $t_j:=(3+3M)^jR$ for $j\in\N$. For any $j\in\N$ and $t_j\le t\le 2t_j$, we have
\EQ{
 t_{j+1}-t \ge t_{j+1}-3t_j = 3Mt_j \ge (t+R)M.}
We claim that if $|T_1-T_2|/R$ is large enough (depending on $M$ and $\nu'$), then for some $j$ with $t_{j+1}<T_2$, we have
\EQ{ \label{small quant}
 \|u\|_{(X\cap K)(t_j,t_{j+1})} + \inf_{t_j<t<2t_j}\|\LR{x}^{-1}u(t)\|_{L^2} \le \nu'.}
Denote by $N_1$ (resp. $N_2$) the number of $j$ for which $t_{j+1}\le T_2$ and the first (resp. the second) term is bigger than $\nu'/2$.
Since $X$ and $K$ consist of $L^8_t$ and $L^4_t$, we have
\EQ{
 \nu' N_1^{1/8} \lec \|u\|_{(X\cap K)(T_1,T_2)} \lec 1.}
On the other hand, by the generalized Morawetz, we have \cite[Lemma 5.3]{2Dsubcrit}
\EQ{
 \int_{\R} \left\|\frac{u(t)}{\LR{x}}\right\|_{L^2}^6 \frac{dt}{1+|t|} \lec 1,}
so we have $(\nu')^{6}(3+3M)^{-1} N_2  \lec 1$.
Therefore we have \eqref{small quant} for some $j$ if $|T_1-T_2|/R\ge (3+3M)^J$ for some $J\gg\ka^{-8}+M(\nu')^{-6}$.
Let $S_0:=t_j$, $S:=t_{j+1}$ and $R':=S_0+R$.

Now we define $v$. Let $\chi\in C_0^\I(\R^2)$ satisfy $0\le\chi\le 1$, $\chi(x)=1$ for $|x|\le 1$ and $\chi(x)=0$ for $|x|\ge 2$.
Let $\chi_a(x):=a\chi(x/R')$. Then by the continuity in $a$, there exists $a\in(0,1]$ such that
\EQ{
 \int_{\R^2} \chi_a^2 e(u;S_0) dx = \nu^2.}
Let $v$ be the free solution satisfying
$$
 (v(S_0,x),\dot v(S_0,x))=\chi_a(x)(u(S_0,x),\dot u(S_0,x)).
$$
Then by using \eqref{H1 cutoff}, we have
\EQ{
 \|\na v(S_0)\|_{L^2}^2
 = \int \chi_a^2|\na u(S_0)|^2 - \chi_a\De\chi_a |u(S_0)|^2 dx,}
and the second term is bounded by
\EQ{
  \|\LR{x}^{-1}u(S_0)\|_{L^2}^2 \le (\nu')^2,}
so we have
\EQ{ \label{v energy}
 E_0(v) = \nu^2 + O((\nu')^2) \le 2\nu^2,}
by choosing $\nu'\ll\nu$.

On the other hand, by the sharp decay estimate \eqref{LIL2 decay}, we have
\EQ{ \label{sup decay est}
 \|v\|_{L^\I B^0_{\I,2}(S,\I)} \lec (R'/|S-S_0|) E_0(v)^{1/2} \lec M^{-1}\nu,}
and by interpolation (cf. \eqref{intp1}),
\EQ{
 \|v\|_{X(S,\I)}
  &\lec \|v\|_{L^\I((S,\I),B^{-1/4}_{\I,\I})}^{1/2}
   \|v\|_{L^4((S,\I),B^{1/4}_{\I,2})}^{1/4}\|v\|_{L^4((S,\I),B^{1/2}_{4,2})}^{1/4}\\
  &\lec (R'/|S-S_0|)^{1/2}E_0(v) \lec M^{-1/2}\nu.}
Hence we get $\|v\|_{L^\I B^0_{\I,2}(S,\I)+X(S,\I)}\le\ka$ by choosing $M\gg\ka^{-2}$.

Now, letting $w=u-v$ and arguing as before we have
\EQ{
 E(w,S_0)&\leq \int(1-\chi^2_a)e_N(u,S_0) + (1-\chi_a)\De\chi_a |u(S_0)|^2\;dx\\
 &\leq E(u)-\nu^2 + C\ka^2.}
Let $2\nu_0(A)<(1-A)/2$, then we can apply our nonlinear estimate to $w$ because
\EQ{
 \|w\|_{H_{loc}(S_0,S_1)}\le \|u\|_{H_{loc}(S_0,S_1)}+\|v\|_{H^1} \le (A+1)/2 < 1.}
By the energy identity together with \eqref{nst-crt}, we have
\EQ{ \label{energy increment}
 E(w,S) &= E(w,S_0) + \int_{S_0}^{S} 2\LR{f(w)-f(u) \mid \dot w}_x dt \\
 &\le E(w,S_0)+2\|f(w)-f(u)\|_{L^1L^2(S_0,S)}\|\dot w\|_{L^\I L^2(S_0,S)}\\
 &\le E(w,S_0)+C(A)\Big(\|u\|_{Y_2(S_0,S)}^{4}+\|v\|_{Y_2}^{4}\Big)(\|u\|_{Y_1(S_0,S)}+\|v\|_{Y_1}).}
By \eqref{small quant} and the Strichartz we have
\EQ{
 \|u\|_{(X\cap K)(S_0,S)} \lec \nu', \quad \|v\|_{X\cap K} \lec E_0(v)^{1/2} \lec \nu.}
Then by the interpolation \eqref{intp1}, we have, using $\nu'\ll\nu$,
\EQ{
 \|u\|_{Y_1(S_0,S)} + \|v\|_{Y_1} \lec \nu^{8\de}, \quad
 \|u\|_{Y_2(S_0,S)} + \|v\|_{Y_2} \lec \nu^{1-\de}.}
Plugging them into \eqref{energy increment}, we obtain
\EQ{
 E(w,S) &\leq E(w,S_0) + C(A)\nu^{4+\de}
  \leq E(u)-\nu^2 + C(\nu')^2 + C(A)\nu^4\\
  &\le E(u) - \nu^2/2,}
by choosing $\nu_0(A)$ sufficiently small.
\end{proof}


\subsection{Perturbation theory}
\begin{lem}
\label{PT}
Let $u$ and $w$ be global solutions to \eqref{NLKG} with $E(u)\le 1$ and $E(w)\le 1$,
and let $v$ be the free solution with the same initial data as $u-w$ at some $t=S$.
For any $M\in(0,\infty)$ and $A\in(0,1)$, there exists
$\kappa(A,M)\in(0,1)$ and $C(A,M)\in(0,\I)$ such that if
\EQ{ \label{v smallness}
 \pt \|w\|_{H_{loc}(S,\I)} \le A,
 \pq \|v\|_{L^\I B^0_{\I,2}(S,\I)} \le \ups(A)/2,
 \pr \|w\|_{X(S,\I)}\le M,
 \pq \|v\|_{X(S,\I)}\le\ka(A,M),}
then we have $\|u\|_{X(S,\I)}\le C(A,M)$.
The above $\ups$ is as given in Lemma \ref{critical estimate}.
\end{lem}
In the subcritical case, we may replace the smallness of $v$ in $B^0_{\I,2}$ by something like $\|v\|_{H^1_{loc}}\le (1-A)/4$.
\begin{proof}
Using the interpolation \eqref{intp1} and $\|v\|_{X(S,\I)}\le\kappa$, we have $\|v\|_{Y(S,\I)}\le C_\de(\ka^{8\de}+\ka^{1-\de}):=\kappa'$.
Let $\eta\in(0,\eta_0(A))$ and $S:=T_0<T_1<\cdot\cdot\cdot<T_N<T_{N+1}=\I$ such that
\EQ{
 \y/2 \le \|w\|_{X(T_j,T_{j+1})}\leq\eta}
for all $j=0,\dots,N$.
Then we have $N^{1/8}\eta\lec M$. The actual size of $\y=\y(A)$ and $\ka=\ka(A,M)$ will be determined later.
Again, using \eqref{intp1} we have
\EQ{
 \|w\|_{Y(T_j,T_{j+1})}\leq C_\de(\y^{8\de}+\y^{1-\de}):=\eta'.}

Let $\Gamma=u-v-w$. Then $\Gamma$ satisfies
\EQ{\label{eqGamma}
 \Gamma(t)=\Gamma_j(t)+\int_{T_j}^tU(t-s)\Big(f(w)-f(u)\Big)(s)ds,}
where $\Gamma_j$ is the free solution with the same initial data as $\Gamma$ at $t=T_j$.

Let $I=(T_j,T)$ with $T>T_j$.
The Strichartz estimates together with \eqref{intp1} and Corollary \ref{crit nonlin} give
\EQ{\label{modStri}
 \|\Ga-\Gamma_j\|_{(Y\cap H)(I)}
  &\le C_0\|f(w)-f(u)\|_{L^1(I,L^2)}\\
  &\le C_1\Big(\|w\|_{Y(I)}+\|\Gamma+v\|_{Y(I)}\Big)^4\|\Gamma+v\|_{Y(I)},}
where $1\le C_0\le C_1$ depend on $A$, provided that
\EQ{ \label{Ga H1 small}
 \|\Ga\|_{L^\I B^0_{\I,2}(I)} \le \ups(A)/2.}

Denote
\EQ{
 \pt p_j:=\|\Gamma_j\|_{(Y\cap H \cap L^\I B^0_{\I,2})(T_j,T)},
 \pq q_j(T):=\|\Gamma\|_{(Y\cap H \cap L^\I B^0_{\I,2})(T_j,T)}.}
Then $q_j$ is continuous, and from \eqref{modStri} we have the following
\EQ{
 &p_0=q_j(T_j)=0,\\
 &q_j(T)\leq p_j+C_1\Big(q_j(T)+\kappa'+\eta'\Big)^4\Big(q_j(T)+\kappa'\Big),\\
 &p_{j+1}\leq p_j+C_1\Big(q_j(T_{j+1})+\kappa'+\eta'\big)^4\Big(q_j(T_{j+1})+\kappa'\Big),}
as long as we keep \eqref{Ga H1 small}.

Now we fix $\eta=\eta(A)$ so small that we have
\EQ{
 C_1(3\eta')^4\le 1/8, \quad  \y'\le \ups(A)/2,}
and then fix $\ka=\ka(A,M)$ small enough to have
\EQ{
 2^{N+1}\ka'< \eta'.}
If $p_j\le 2^j\ka'$ and $q(T)\le 2^{j+1}\ka'$, then we have
\EQ{
 q_j(T) \leq 2^j\ka' + C_1(3\eta')^42^{2+j}\ka'\le \frac{3}{2}2^j\ka',}
hence by continuity we have $q(T)\le 2^{j+1}\ka'$, and in the same way for the limit $T\to T_{j+1}-0$,
we get $p_{j+1}\le 2^{j+1}\ka'$.

Hence by induction, we obtain $p_j\leq2^j\kappa'<\eta'$
and $q_j(T_{j+1})<\eta'$ for any $j\leq N$.
Using the interpolation \eqref{intp1}, we conclude that
\EQ{
 \|u\|_{X(S,\I)}\lec \|u\|_{Y(S,\I)}^{1/2} \lec N^{1/8}\eta'\leq C(A,M)}
as desired.
\end{proof}


\section{Uniform global bound in the sub-critical case}
In the sequel, we want to show the following result which is the crucial step in proving Theorem \ref{Main} in the sub-critical case.

\begin{prop}
There exists an increasing function
$C:[0,1)\rightarrow [0,\infty)$ such that
for any $0\leq E<1$, any global solution $u$ of
the \eqref{NLKG} with $E(u)\leq E$ satisfies
\begin{eqnarray}
\label{globou}
\|u\|_{X(\mathbb R)}\leq C(E).
\end{eqnarray}
\end{prop}

\begin{proof}
First, note that arguing as in Lemma \ref{bump}, one can show that estimate \eqref{globou} is satisfied if the energy is small. Denote by
\EQ{E^*:=
\sup\{0\leq E<1:\;    \sup_{E(u)\le E} \|u\|_{X(\mathbb R)}<\I\}.}

From the small data scattering, it is clear that $E^*>0$.
The goal is to show that $E^*=1$. Assume that $E^*<1$.
Then for any $E\in(E^*,1)$ and any $n\in(0,\I)$, there exists a global solution $u$ such that $E(u)=E$ and $\|u\|_{X(\mathbb R)}> n$.
We are going to show that if $E$ is close enough to $E^*$, then $n$ cannot be arbitrarily large.

By time translation, we may assume that $\|u\|_{X(0,\I)}>n/2$.
Since $\|u\|_{H_{loc}(\R)}\le \|u\|_{H(\R)}\le E<1$, we can apply all the lemmas in the previous sections with $A:=E$.

First we fix $\y=\y_0(A)/2$, $R=R(\y)$ and $\nu=\min(\nu_0(A),\sqrt{C_2}\y^8)/2$, where $\y_0(A)$, $R(\y)$, $\nu_0(A)$ and $C_2$ are given in Lemmas \ref{bump} and \ref{STLE}.

Next we choose $E\in(E^*,E^*+\nu^2/4)$ such that $E-\nu^2/2 < E^*-\nu^2/4=:E'$.
Then by definition of $E^*$, there exists $M<\I$ such that $\|w\|_X\le M$ for any global solution $w$ with $E(w)\le E'$.
We fix $\ka=\min(\ups(A),\ka(A,M))/4$, where $\ups$ and $\ka$ are given in Corollary \ref{crit nonlin} and Lemma \ref{PT}, respectively.
Then we fix $L=L(\nu,\ka)$ which was given in Lemma \ref{STLE}.

Now let $0=\t_0<\t_1<\cdots$ such that $\|u\|_{X(\t_{j-1},\t_j)}=\y/2$, and let $J\in\N$ be the smallest satisfying $|\t_J-\t_{J+1}|\ge LR$. Then Lemma \ref{long} gives an upper bound on $J$:
\EQ{
 1 \gec \sum_{j=1}^J \frac{\y^{40}}{jLR+\y^{-3}} \sim \frac{\y^{43}}{L}\log J.}

Applying Lemma \ref{bump} on $(\t_{J-1},\t_J)$, we get some $T_1\in(\t_{J-1},\t_J)$, $c\in\R^2$ such that with $T_2=\t_{J+1}$, all the assumptions for Lemma \ref{STLE} are fulfilled.
By using its conclusion, let $w$ be the global solution with the same initial data as $u-v$ at some $t=S\in (T_1,T_2)$.
Since $E(w)\le E-\nu^2/2<E'$, we have $\|w\|_X\le M$. Also $\|w\|_{H_{loc}}\le E(w)\le E=A$.
Thus all the assumptions in Lemma \ref{PT} hold, so we obtain
\EQ{
 \|u\|_{X(S,\I)}\le C(A,M),}
Since $S\in(T_1,T_2)$, we get a bound on $n$ by using the estimate on $J$ as well:
\EQ{
 n/2 \le \|u\|_{X(0,\I)} \le (J+1)^{1/8}\y+C(A,L),}
which is a contradiction.
\end{proof}


\section{Scattering in the critical case}
Let $u$ be a solution of \eqref{NLKG} with $E(u)= 1$, and fix it. We denote the concentration radius of energy $1-\e$ at time $t$ by
\EQ{
 r_\e(t) = R_{1-\e}[u](t).}
It is easy to see by finite propagation that
\EQ{ \label{Lip}
 |r_\e(t)-r_\e(s)| \le |t-s|,}
for any $\e,s,t$. If we have the scattering, then
$\displaystyle\lim_{t\to\I} r_\e(t)=\I$ for any $\e>0$. We will reduce the problem to the subcritical result with
careful investigation on the evolution of $r_\e(t)$ case by case.
The most problematic case is that where $r_\e(t)$ is neither bounded from above nor
away from zero, then we will perform the separation of concentration energy in a similar way as above, but using the oscillatory behavior of the concentration radius.
The use of the critical Besov space in the nonlinear estimate is essential only in that case.


\subsection{Dispersive and non-dispersive cases}
First we deal with the two extreme cases.
Let $\e\in(0,1)$. By Lemma \ref{crit-nlst} together with \eqref{equiv}, we can control the nonlinearity on any interval $I$ satisfying
\EQ{ \label{Str rad}
 \inf_{t\in I} r_\e(t)\ge 6,}
with constants depending only on $A=1-\e$.
Therefore, if we have
\EQ{
 \liminf_{t\to\I} r_\e(t) > 6,}
for some $\e\in(0,1)$, then the scattering follows in the same way as in the subcritical case.
More precisely, here we need two different $A$'s, namely $A_0=1-\e$ before the perturbation lemma, and $A_1=\sqrt{1-\nu(A_0)^2/4}$ for the nonlinear solution $w$ with reduced energy.
Indeed, we would fall into a vicious circle if we would insist a single $A$, since $1-A_1\ll 1-A_0$.
Specifically, those parameters are defined by
\EQ{ \label{double A def}
 \pt \y=\y_0(A_0)/2, \pq \nu=\min(\nu_0(A_0),\sqrt{C_2}\y^8)/2,
 \pr A_1 = \sqrt{1-\nu(A_0)^2/4},\pq \ka=\min(\ups(A_1),\ka(A_1,M))/4, \pq L=L(\nu,\ka),}
where $M<\I$ is given by the subcritical result that we have $\|w\|_X\le M$ for any nonlinear solution $w$ satisfying $E(w)\le A_1^2$.
Thus using Lemmas \ref{bump} and \ref{STLE} with $A=A_0$ and $R=R(\y)$, and Lemma \ref{PT} with $A=A_1$,
we deduce that $\|u\|_{X(0,\I)}<\I$, which implies the scattering for $u$.

Next consider the case where we have for all $\e\in(0,1)$
\EQ{ \label{nondisp}
 \ti r_\e:=\limsup_{t\to\I} r_\e(t) + 1 < \I,}
namely the solution does not disperse at all as $t\to\I$. By the definition of $r_\e(t)$, there exists a function $c:[0,\I)\to\mathbb R^2$ such that
\EQ{
 \int_{|x-c(t)|>r_\e(t)} e_N(u,t) dx \le \e,}
and the above assumption \eqref{nondisp} implies that
\EQ{\label{I}
 \limsup_{t\to\I} \int_{|x-c(t)|>\ti r_\e} e_N(u,t) dx \le \e.}
The finite propagation property implies that we can choose $c(t)$ such that
\EQ{\label{D}
 \limsup_{t\to\I} |c(t)|/t +1 \le M \lec 1,}
with $M$ independent of $\e$.
Then we obtain using the H\"older inequality and $e_N(u)\ge |u|^2$, that for $t$ large enough
\EQ{ \label{local L2 dominated}
 \|u(t)\|_{L^2} &\lec \e^{1/2} + \|u(t)\|_{L^2(|x|<M|t|)}\\
  &\lec \e^{1/2} + (\ti r_\e)^{1-2/p}\|u(t)\|_{L^p_x(|x|<M|t|)},}
for any $p>2$. On the other hand, the generalized Morawetz inequality \eqref{ME} 
implies that
\EQ{
 \lim_{T\to\I} \frac{1}{T} \int_{T}^{C T} \int_{|x|\le Mt} \G(u) dx dt =0}
for any $C\in (1,\I)$.
Combining this with \eqref{local L2 dominated}, we obtain
\EQ{
\label{L2}
 \lim_{T\to\I} \frac{1}{T} \int_T^{C T} \int_{\mathbb R^2} |u|^2 dx dt = 0.}
By using this and the inversional identity as in Lemma \ref{II} 
, we get
\EQ{
\label{E}
 \lim_{t\to\I} \int_{|x|\le Mt} Q(u) dx = 0,}
where $Q(u)$ is as defined in \eqref{Q}.

Now suppose that
\EQ{
 \lim_{n\to\I} \|\na u(t_n)\|_{L^2} = 1}
for some sequence $t_n\to\I$. Then $E(u)=1$ implies that
\EQ{
 \lim_{n\to\I} \|\dot u(t_n)\|_2^2 + \|u(t_n)\|_2^2 \to 0,}
and hence
\EQ{
 \lim_{n\to\I} \int_{|x|\le Ct_n} Q(u(t_n)) dx \ge 1,}
contradiction. Therefore we have
\EQ{
 \limsup_{t\to\I} \|\na u(t)\|_{L^2} <1.}
On the other hand, \eqref{E} together with \eqref{D} and \eqref{I} implies
\EQ{
\limsup_{t\to\I} \|u(t)\|_{L^2} =0.}
Hence we have
\EQ{
 \limsup_{t\to\I} \|u(t)\|_{H^1} < 1,}
so that we can treat this case in the same way as the above dispersive case, with $A_0$ between $1$ and the left hand side, and the other parameters given by \eqref{double A def}.
However, we get a contradiction in the end between the scattering and the assumption \eqref{nondisp}.


\subsection{Waving concentration case}
This is the  wildest case. We have
\EQ{ \label{asy conc}
 \liminf_{t\to\I} r_\e(t) \le 6}
for all $\e\in(0,1)$, and
\EQ{ \label{asy disp}
 \limsup_{t\to\I} r_\e(t) = \I,}
for some $\e\in(0,1)$; namely the solution repeats dispersing and regathering infinitely many times.
Now we fix $\e\in(0,1/4)$ for which we have \eqref{asy disp}.


\subsection{Extracting very long dispersive era}
By \eqref{asy conc}, there exists $\T_0>0$ such that
\EQ{ \label{preconc}
 r_{1/4}(\T_0) \le 6.}
Let $B:(0,1/4]\to[18,\I)$ be a continuous function of $\e$. \eqref{asy disp} supplies $\T>\T_0$ satisfying
\EQ{ \label{choose T}
 r_\e(\T) \ge B(\e).}
The actual form of $B$ will be determined later.

By \eqref{choose T}, the Lipschitz continuity \eqref{Lip}, preceding
concentration \eqref{preconc} and the assumption $B(\e)/2\ge 9>6$,
there exist $I=(\T_1,\T_2)$ such that
\EQ{
 &\T_0 < \T_1 < \T < \T_2,\quad |\T_2-\T_1|\ge B(\e),\quad
  r_\e(\T_1)=6,}
and
\EQ{
 [\T_1,\T_2]\ni\forall t,\quad r_\e(t)\ge 6,}
which allows us to use the Strichartz estimate in $I$ depending only on $A:=1-\e$.

Now the idea is to argue as in the subcritical case inside $I$.
After separating the concentrated energy into a free solution $v$ and another nonlinear solution $w$ with reduced energy,
we can apply the perturbative argument beyond $\T_2$, thanks to the decay of the free solution and Corollary \ref{crit nonlin}.
Indeed, this is the only place we essentially need that version of nonlinear estimate, otherwise the version with $H^1_{loc}$ would suffice.

Note that $v$ does not decay uniformly in $H^1_{loc}$.
Indeed, if we consider a free wave $v$ whose frequency is supported around $|\x-Nc|\ll 1$ with $N\gg 1$ and $|c|=1$, then $v(t,x)$ remains essentially unchanged around $x=tc$ for $|t|\ll R$, and so $H^1_{loc}$ does not decrease during that period.
For such high frequency free waves, we can say that $\|v\|_{H^1}^2$ is reduced almost to $E_0(v)/2$ after some time (independent of $N$) due to the equipartition of energy, but it does not imply any decrease on $\|u\|_{H^1_{loc}}$; without any further information on $w=u-v$, it can still approach $1$ arbitrarily closely at later times.
On the other hand, when measured in the Besov space $B^0_{\I,2}$, such a high-frequency wave packet is small from the beginning because of its concentration in high frequency.


Now we set up the assumptions for Lemma \ref{STLE}. As before we fix $A_0=1-\e$, and the other parameters are given by \eqref{double A def}.

\subsection{Temporary scattering case}
First we consider the case $\|u\|_{X(I)}\le2\y$.
The Strichartz norm in $I$ may be too small to apply Lemma \ref{bump}.
Hence we just choose $T_1=\T_1$, $T_2=\T_2$. By the definition of $r_\e$, there is $c\in\R^2$ such that
\EQ{ \label{energy bump at first}
 \int_{|x-c|<R} e_N(u,T_1) dx \ge \nu^2,}
with $R=6$.
Setting $B(\e)\ge 6L$, we get all the assumptions of Lemma \ref{STLE} with $A=A_0$.


\subsection{Multilayer case}
Next we consider the remaining case $\|u\|_{X(I)}>2\y$. Then we split $I$ into subintervals by $\T_1=\t_0<\t_1<\cdots<\t_N=\T_2$, with some $N\ge 3$ such that
\EQ{
 \|u\|_{X(\t_{j-1},\t_j)} = \y,\quad j=1,2,\ldots N-1.}
Then Lemma \ref{long} implies that
\EQ{
 \sum_{j=1}^N \frac{\y^{40}}{\t_j-\t_0+\y^{-3}} \lec 1.}
If we have $|\t_j-\t_{j-1}|\le LR(\y)$ for all $j=1,\dots N$, then
\EQ{
 1\gec \sum_{j=1}^N \frac{\y^{43}}{jL+1} \gec \y^{43}\log N,}
and so $|\T_2-\T_1|\le LRN \lec LRe^{C\y^{-43}}$ with some constant $C\ge 1$.

Hence if we choose $B(\e)\gg LR(\y)e^{C\y^{-43}}$, then there exists $j\in\{0,\dots,N-1\}$ such that $|\t_j-\t_{j+1}|\ge LR(\y)$.
If $j=0$, then we choose $T_1=\T_1=\t_0$ and $T_2=\t_1$. As in the previous case, there exists $c\in\R^2$ with \eqref{energy bump at first} with $R=6$.
Since $R(\y)\gg 6$, we get all the assumptions of Lemma \ref{STLE} with $A=A_0$.

If $j\ge 1$, then we choose $T_2=\t_{j+1}$, and
Lemma \ref{bump} implies that there exists $T_1\in(\t_{j-1},\t_j)$ and $c\in\R^2$ such that
\EQ{
 \int_{|x-c|<R} e_N(u,T_1) dx \ge \nu^2,}
with $R=R(\y)$, hence all the assumptions of Lemma \ref{STLE} are fulfilled with $A=A_0$.

Thus we can use Lemma \ref{STLE} with $A=A_0$ in all cases,
hence we get the scattering of $u$ as before, even though it contradict the assumption \eqref{asy conc}.


\subsection{Global Strichartz estimate}


We end this section with a remark about global  Strichartz bounds.
\begin{cor}
The global  solution  $u$ of theorem \ref{nlkg} satisfies the following
global bounds for any $\e\in(0,1)$.
\EQ{
  \| u  \|_{X(\mathbb R)} + \| f(u)  \|_{L^1(\mathbb R;L^2(\R^2)) } \leq C_\e  (\inf_{t\in \mathbb R  } r_\e(t) ).}
\end{cor}

The proof of this corollary   is left to the reader. One has just to be
more quantitative in the estimates in the previous subsections.
Moreover, one  has also to    use  the following  lemma,  which
yields that Lemma \ref{crit-nlst} still
holds if \eqref{H6} is replaced by \eqref{th-r}.
\begin{lem} \label{impro-MT}
For any $\th\in(0,1) $ and $r > 0$, take $\lambda  $ close to 1 such
that
$$ (\sqrt{\theta\la} + 4\sqrt{2(\la-1)}/r  )^2
 + 2(\la-1) < 1 < \la .
$$
There exists $C$ depending only on  $\th,r $ and $\lambda$ such that
if
\EQ {  \| \nabla u\|_{L^2}^2 +  \| u\|_{L^2}^2 \leq 1 \quad
\hbox{and}\quad   \| u \|_{H^1[r]} \leq \theta     \label{th-r},  }
 then
\EQ{ \label{conc-impr}  \int_{\R^2}  (e^{4\pi \lambda |u|^2} -1) dx \leq C \|u\|_{L^2}^2 .
  }
\end{lem}

\begin{proof}
By contradiction. Assume (\ref{conc-impr}) does not hold, then
there exists a sequence  $u_n$ such that
$\| u_n \|_{H^1} \leq 1 $,   $ \| u_n \|_{H^1[r]} \leq \theta     $
and
$$   \int  (e^{4\pi \lambda |u_n|^2} -1) \geq n \|u_n\|_{L^2}^2 .   $$
Hence, necessarily
\EQ{  \liminf_{n\to \infty}   \| \nabla u_n\|_{L^2}^2 +  \frac12 \| u_n\|_{L^2}^2
  =: \alpha \geq \frac{1}{\lambda},   }
otherwise, we get a contradiction by applying the Trudinger-Moser
inequality  (see for instance \cite[Proposition 1]{2Dglobal}).

Hence, we deduce that
\EQ{  \limsup_{n\to \infty}     \frac12 \| u_n\|_{L^2}^2
   \leq 1 - \frac{1}{\lambda}.   }

Take $\phi$ a cut-off function such that
$\phi = 1$ for $|x| \leq r/2$ and $\phi = 0$ for $|x| \geq r$, $0\leq \phi  \leq 1$
and $ \| \nabla \phi\|_{L^\infty}  \leq 4/r$. Hence,
\EQ{
 &\limsup_{n\to \infty}  \|\nabla (\phi u_n) \|_{L^2}^2
   +  \|\phi u_n \|_{L^2}^2 \\
 &\leq  \limsup_{n\to \infty}
 (   \| \phi \nabla  u_n  \|_{L^2} +    \|  u_n \nabla  \phi \|_{L^2}    )^2
 + \|\phi u_n \|_{L^2}^2   \\
 &\leq   (\sqrt \theta   + \sqrt{2(1-1/\lambda)}4/r )^2
 + 2(1-1/\lambda)  < 1/\lambda.}
Hence, by Trudinger-Moser, we deduce that for $n$ big enough
\EQ{ \label{conc-impr1}  \int  (e^{4\pi \lambda |\phi u_n|^2} -1) \leq C \|\phi u_n\|_{L^2}^2 .
  }
We can cover $\R^2$ by balls of radius $r/2$ in such a way that each
point is in at most 10 balls of the same centers and radius $r$.
Adding the estimates (\ref{conc-impr1})   together, we deduce that
\EQ{ \label{conc-impr2}  \int  (e^{4\pi \lambda | u_n|^2} -1) \leq 10  C \| u_n \|_{L^2}^2 .
  }
which gives a contradiction. Hence (\ref{conc-impr}) holds.

\end{proof}

\section{Criticality of the nonlinear estimate by the Strichartz norms}
We see that without the assumption on the local $H^1$ norm, our nonlinear estimate does not hold generally in the critical case.
\newcommand{\rh}{\rho}
\begin{prop} \label{prop:ctex NLKG}
For any $\de>0$, there exists a sequence of radial free Klein-Gordon solutions $v_N$ ($N\to\I$) such that
\EQ{
 \pt \int_{\R^2} |\dot v_N(t)|^2 + |\na v_N(t)|^2 + |v_N(t)| dx<1, \pq E(v_N,0) \le 1 +\de,
 \pr \|f(v_N)\|_{L^p_t L^q_x (|t|\ll N^{-1},|x|\ll N^{-1})} \ge C_\de (\log N)^{1/2},}
for any $p,q\in[1,\I]$ satisfying $1/p+2/q=2$.
\end{prop}
To apply the Strichartz estimate to the nonlinear term, we have to put it in some $L^p H^{\s,q}$ satisfying at least
\EQ{
 \frac{1}{p}+\frac{1}{q}\ge \frac{3}{2},\quad \s\ge \frac{1}{p} + \frac{2}{q} -2,}
which is embedded into $L^p L^\r$ with $1/p+2/\r=2$.
Thus the above example implies that in the critical case we cannot control the nonlinearity just by the Strichartz estimate even for the free solutions, which is usually the first step to construct solutions by the iteration argument.
However, it does not imply any sort of weak ill-posedness (such as singularity of the solution map) in the critical case, because we can not make the nonlinear energy critical.
\begin{proof}
Let $a>1/2$ and define $v_N$ by the initial data using the Fourier transform
\EQ{
 \dot v_N(0)=0,\quad v_N(0) = \sqrt{\frac{2\pi}{\log N}} \frac{1}{(2\pi)^2} \int_{1<|\x|<e^{-a}N} |\x|^{-2}e^{i\x x}d\x.}
By Plancherel
\EQ{ \label{H1 L2 bd}
 \pt \|\na v_N(0)\|_{L^2}^2 = \frac{2\pi}{\log N} \frac{1}{(2\pi)^2}\int_{1<|\x|<e^{-a}N}|\x|^{-2} d\x
  = \frac{\log N-a}{\log N},
 \pr \|v_N(0)\|_{L^2}^2 = \frac{2\pi}{\log N} \frac{1}{(2\pi)^2}\int_{1<|\x|<e^{-a}N}|\x|^{-4} d\x
  < \frac{1}{2\log N}.}
By the sharp Trudinger-Moser inequality \eqref{Trudinger-Moser}, there exists $M>0$ such that for any $\mu>0$
\EQ{
 \sup_{\|\na \psi\|_{L^2}^2 + \mu\|\psi\|_{L^2}^2 \le 1} \int F(\psi) dx \le M/\mu,}
where $\mu$ can be removed or inserted by rescaling. Then \eqref{H1 L2 bd} implies that
\EQ{
 \pt \int 2F(v_N(0)) dx \le \frac{M}{a},
 \pq E_0(v_N) < 1, \pq E(v_N,0) < 1 + \frac{M}{a},}
so that we get the desired nonlinear energy bound on $v_N$ by choosing $a\ge M/\de$.

Next, the free solution is given by Fourier transform
\EQ{
 v_N(t,x) = \sqrt{\frac{2\pi}{\log N}}\frac{1}{(2\pi)^2} \int_{1<|\x|<e^{-a}N} |\x|^{-2}\cos(-t\LR{\x}+\x x) d\x,}
and hence in the region where $t\sim\e N^{-1}$ and $|x|\sim\e N^{-1}$ for some $0<\e\ll 1$,
\EQ{
 \sqrt{\frac{\log N}{2\pi}} (2\pi)^2 v_N(t,x) \ge   \sum_{k=1}^{K}\int_{e^{k-1}<|\x|<e^k}|\x|^{-2}\cos(-t\LR{\x}+\x x) d\x,}
where $K$ is the maximal integer satisfying
\EQ{
  K \le \log N-a.}
The cosine is bounded in the region by
\EQ{
 \cos(-t\LR{\x}+\x x) \ge 1 - \frac{\e^2}{2} e^{2(k-K)},}
so the above integral is bounded from below
\EQ{
 \ge 2\pi\left[\log N - a - \frac{\e^2}{2} \sum_{k=1}^K e^{2(k-K)} \right],}
and the negative part is bounded by
\EQ{
 -a-2\pi \frac{\e^2}{2}\frac{1}{1-e^{-2}} \ge -2\pi(a+\e^2).}
Thus we obtain
\EQ{
 \pt v_N(t,x) \ge \sqrt{\frac{\log N}{2\pi}} - \frac{a+\e^2}{\sqrt{2\pi\log N}},
 \pr e^{4\pi|v_N|^2}|v_N|^2 \gec N^2\log N,}
when $(a+\e^2)^2<\log N$.

Thus we conclude
\EQ{
 \pt \inf_{t\sim\e N^{-1}} \|e^{4\pi|v|^2}v\|_{L^q_x(|x|\sim\e N^{-1})} \gec N^2\sqrt{\log N} (\e N^{-1})^{2/q},
 \pr \|e^{4\pi|v|^2}v\|_{L^p_t L^q_x( t \sim \e N^{-1}, |x|\sim\e N^{-1})} \gec \e^{1/p+2/q}N^{2-2/q-1/p}\sqrt{\log N}
 \sim \e^2 \sqrt{\log N}.}
\end{proof}


\section{The case of {\sf NLS}}

For the sake of clarity, we recall the 2D nonlinear Schr\"odinger
equation ({\sf NLS}):
\begin{equation}
\nonumber
\left\{
 \begin{aligned}
  &i\,\dot u+\Delta u = \left(e^{4\pi |u|^2}-1-4\pi |u|^2\right)u=f(u),\quad
   u:\R^{1+2}\to\C, \\
  &u(0,x) = u_0 (x)\in H^1(\R^2),
 \end{aligned}
\right.
\end{equation}
and the conserved quantities
\begin{equation}
\nonumber
M(u,t)=\int_{\R^2}|u|^2  dx,
\end{equation}
\begin{equation}
\nonumber H(u,t)=\int_{\R^2}\,\Big( |\na u|^2 + 2F(u)\Big)\, dx,
\end{equation}
where $$
 F(u) = \frac{1}{8\pi}\left(e^{4\pi |u|^2}-1 - 4\pi
 |u|^2-8\pi^2|u|^4\right).
 $$
Also we define
\EQ{
 E(u,t) := H(u,t) + M(u,t).}
For a time slab $I\subset\R$, we define $S^1(I)$ via
$$
\|u\|_{S^1(I)} = \|u\|_{L^\I(I,H^1_x)} +  \|u\|_{L^4(I,H^{1,4}_x)}.
$$
By the Strichartz estimates we have
\EQ{ \label{SStr}
  \|u\|_{S^1}\lesssim \|u(0)\|_{H^1}+\|\langle\nabla\rangle (i\,\dot u
   + \Delta u)\|_{L^{\frac{2}{1+2\eta}}(L^{\frac{1}{1-\eta}}_x)},}
for any $0<\eta\leq 1/2$.

The scattering result Theorem \ref{Main-NLS} is easily proved by the following two lemmas:
First we have the Strichartz-type estimate on the nonlinearity
\begin{lem}
\label{ST-NE} For any $H\in(0,1)$, there exists $\de\in(0,1)$, such that for
any time slab $I$, any $T\in I$ and any $H^1$ solution $u$ of \eqref{NLS} with $H(u)\le H$, we have
$$
\|u\|_{S^1(I)} \lec
  \|u(T)\|_{H^1}+\|u\|_{L^4(I,L^8)}^{4\delta}\|u\|_{S^1(I)}^{5-4\delta},
$$
\end{lem}

Next we have a global {\it a priori} bound.
It was proved independently by Planchon-Vega \cite{PV} and Colliander et al. \cite{CGT}
\begin{lem}
\label{GB} Let $u$ be a global solution of \eqref{NLS} in $H^1$. Then
\EQ{ \nonumber
 \|u\|_{L^4(\R,L^8)} \lec \|u\|_{L^\I(\R;L^2)}^{3/4} \|\na u\|_{L^\I(\R;L^2)}^{1/4} \lec M(u)^{3/8}H(u)^{1/8}.}
\end{lem}
Actually both of them gave a priori bound on some Sobolev norm on $|u|^2$. The above is a consequence of it via the Sobolev embedding.

By the above global bound, we can decompose $\R$ into a finite number of intervals on which the $\|u\|_{L^4 L^8}$ norm is sufficiently small. Then the first lemma gives a uniform bound on $\|u\|_{S^1}$ on each interval, and hence by summing it up for all intervals, we obtain a priori bound
\EQ{
 \|u\|_{S^1(\R\times\R^2)} \le C(E(u)) < \I,}
and thereby the scattering for $u$.

\begin{proof}[{\bf Proof of Lemma 8.1}]
It suffices to estimate the nonlinear term in some dual Strichartz
norm as in \eqref{SStr}. Choose $0<\de<1$ and $\lambda>0$ such that
\EQ{ \label{def eta ka}
  K:=\frac{H+1}{2}<1,\pq 2\pi(1+2\de)\lambda K^2=2, \quad
   \lambda>\frac{1}{\pi(1-\de)}.}
We estimate only $\na f(u)$, since the same estimate on $f(u)$ is easier.
Note that
$$ |\nabla f(u)| \lec |\nabla u| |u|^2(e^{4\pi|u|^2}-1).
$$

In the case $\|u\|_{L^\infty_x}\geq K$, we have by the H\"older inequality,
\EQ{
 \|\na f(u)\|_{L^{\frac{1}{1-\de}}_x}
  &\lec \|\na u\|_{L^{\frac{2}{1-\de}}}\|u\|_{L^{\frac{4}{\de}}}^2
    \|e^{4\pi|u|^2}-1\|_{L^1}^{1/2-\de}
    \|e^{4\pi|u|^2}-1\|_{L^\infty}^{1/2+\de}.}
The third term on the right is bounded by the Trudinger-Moser \eqref{Trudinger-Moser}.
For the last term we use the $H_{\mu}$ version of the logarithmic
inequality \eqref{H-mu} with $\mu:=\min(1,\sqrt{(1-H)/M})>0$. Since
$$
 \|u\|_{H_\mu}^2 = \|\nabla u\|_{L^2}^2+\mu^2 \|u\|_{L^2}^2\leq H+\mu^2 M \le
 \frac{H+1}{2}<1,
$$
that term is bounded by
\EQ{
 e^{4\pi(1/2+\de)\|u\|_{L^\I}^2}
  \lec (1 + \|u\|_{C^{1/2-\de/2}}/\|u\|_{H_\mu})^{2\pi(1+2\de)\|u\|_{H_\mu}^2}
  \lec \|u\|_{C^{1/2-\de/2}}^2,}
where we used \eqref{def eta ka} as well as $\|u\|_{C^{1/2-\de/2}}\ge\|u\|_{L^\I}\ge K$.
The case $\|u\|_{L^\I}\le K$ is easy, since then $|\na f(u)|\lec |\na u||u|^4$.

Now we integrate in time using the H\"older to obtain
$$
\|\na f(u)\|_{L^{\frac{2}{1+2\de}}(L^{\frac{1}{1-\de}})}
\lesssim \|\nabla u\|_{L^{\frac{2}{\de}}(L^{\frac{2}{1-\de}})}
\|u\|_{L^{\frac{2}{\de}}(L^{\frac{4}{\de}})}^2
\|u\|_{L^{\frac{4}{1-\de}}(C^{1/2-\de/2})}^2.
$$

Finally, the complex interpolation
and the Sobolev embedding imply that
\EQ{
 &\|\nabla u\|_{L^{\frac{2}{\de}}(L^{\frac{2}{1-\de}})} \lec \|\na u\|_{L^\I L^2}^{1-2\de} \|\na u\|_{L^4 L^4}^{2\de} \lec \|u\|_{S^1},\\
 &\|u\|_{L^{\frac{2}{\de}}(L^{\frac{4}{\de}})}
  \lec \|u\|_{L^\I L^2}^{1-2\de} \|u\|_{L^4 L^8}^{2\de},\\
 &\|u\|_{L^{\frac{4}{1-\de}}(C^{1/2-\de/2})}
  \lec \|u\|_{L^{\frac{4}{1-\de}}(H^{1,\frac{4}{1+\de}})}
  \lec \|u\|_{L^\I H^1}^\de \|u\|_{L^4 H^{1,4}}^{1-\de} \lec \|u\|_{S^1}.}
Plugging them into the above, we deduce the result as desired.
\end{proof}

Finally, we observe that the same example as in Proposition \ref{prop:ctex NLKG} implies that the linear energy and the Strichartz estimate are not sufficient to control the nonlinearity in the critical case.

\begin{prop}
For any $\de>0$, there exists a sequence of radial free Schr\"odinger solutions $v_N$ ($N\to\I$) such that
\EQ{
 \pt \int_{\R^2} |\na v_N|^2 + |v_N|^2 dx <1, \pq H(v_N,0) \le 1 +\de, \pq
 \pr \|\na f(v_N)\|_{L^p_t L^q_x (|t|\ll N^{-2},|x|\ll N^{-1})} \ge C_\de (\log N)^{1/2},}
for any $(p,q)\in[1,\I]$ satisfying $1/p+1/q=3/2$.
\end{prop}
The above norm on $f(v_N)$ is the dual Strichartz norm in $H^1_x$ for the linear Schr\"odinger equation.
\begin{proof}
We take the same initial data as in Proposition \ref{prop:ctex NLKG}:
\EQ{
 v_N(0) = \sqrt{\frac{2\pi}{\log N}} \frac{1}{(2\pi)^2} \int_{1<|\x|<e^{-a}N} |\x|^{-2}e^{i\x x}d\x.}
Then the proof for the previous Proposition gives the desired bounds on the initial data.
Also, it implies that for the free solution
\EQ{
 v_N(t,x) = \sqrt{\frac{2\pi}{\log N}}\frac{1}{(2\pi)^2} \int_{1<|\x|<e^{-a}N} |\x|^{-2}e^{-it|\x|^2+i\x x} d\x,}
we have in the region where $t\sim\e^2N^{-2}$ and $|x|\sim\e N^{-1}$ for some small fixed $\e>0$ and large $N\in\N$,
\EQ{
 \pt \Re v_N(t,x) \ge \sqrt{\frac{\log N}{2\pi}} - \frac{a+\e^2}{\sqrt{2\pi\log N}},
 \pq e^{4\pi|v|^2}|v|^2 \gec N^2\log N.}
In this case we have to estimate $\na v$ also.
By the radial symmetry, it suffices to consider the case $x=(x_1,0)$ and $\na v_N=(\p_1 v_N, 0)$. Then
\EQ{
 \p_1 v_N = \p_r v_N \pt\sim \frac{1}{\sqrt{\log N}}\int_{1<|\x|<e^{-a}N} \frac{\x_1}{|\x|^2}e^{-it|\x|^2+i\x x} d\x
 \pr= \frac{i}{\sqrt{\log N}}\int_1^N e^{-it\rh^2} \int_{-\pi}^{\pi} \cos\th \sin(r\rh\cos\th) d\th d\rh,}
since $0<t\rh^2<\e^2\ll 1$ and $0<r\rh\cos\th<\e\ll 1$, we get
\EQ{
 |\p_r v_N| \sim \frac{1}{\sqrt{\log N}} \int_1^N r\rh d\rh \sim \frac{r N^2}{\sqrt{\log N}} \sim \frac{\e N}{\sqrt{\log N}}.}
Thus we conclude
\EQ{
 \pt \inf_{t\sim\e^2 N^{-2}} \|e^{4\pi|v_N|^2}|v_N|^2\na v_N\|_{L^q_x(|x|\sim\e N^{-1})} \gec N^2\log N \frac{\e N}{\sqrt{\log N}} (\e N^{-1})^{2/q},
 \pr \|\na f(v_N)\|_{L^p_t L^q_x( t \sim \e^2 N^{-2}, |x|\sim\e N^{-1})}
  \gec \e^{2/p+2/q} N^{3-2/q-2/p} \sqrt{\log N} \sim \e^3 \sqrt{\log N}.}
\end{proof}



\begin{thebibliography}{10}
\bibitem{AT}
S.~Adachi and K.~Tanaka, {\em Trudinger type inequalities in $\mathbb R^N$ and their best exponents}, Proc.~Amer.~Math.~Soc. {\bf 128} (2000), no.~7, 2051--2057.

\bibitem{A}
A.~Atallah, {\em Local existence and
estimations for a semilinear wave equation in two dimension
space.} Boll. Unione Mat. Ital. Sez. B Artic. Ric. Mat. {\bf 8} (2004), 1, 1--21.

\bibitem{BG}
H.~Bahouri and P. G\'erard, {\em High frequency approximation of solutions to critical nonlinear wave equations}, Amer. J. Math. {\bf 121} (1999), 131--175.

\bibitem{BS}
{H.~Bahouri and J. Shatah}, {\em Decay estimates for the critical semilinear wave equation}, Ann. Inst. H. Poincar\'e
    Anal. Non-lin\'eaire, {\bf 15} (1998), 783--789.

\bibitem{Bourgain}{J. Bourgain}, {\em Scattering in the
energy space and below for 3D NLS}, J. Anal. Math. {\bf 75} (1998) 267--297.

\bibitem{Bourgain1}{J. Bourgain}, {\em Global wellposedness
of defocusing critical nonlinear Schr\"odinger equation in the radial case},
J. Amer. Math. Soc. {\bf 12} (1999), no.~1, 145--171.

\bibitem{Brenner}{P.~Brenner},
{\em On scattering and everywhere defined scattering operators for nonlinear Klein-Gordon equations}, J. Differential Equations {\bf 56} (1985), 310--344.

\bibitem{Caz}
T.~Cazenave, {\em Equations de Schr\"odinger non lin\'eaires en dimension deux.} Proc.~Roy.~Soc.~Edinburgh Sect. A {\bf 84} (1979), no.~3-4, 327--346.

\bibitem{CazWeis}{T. Cazenave and F.B. Weissler},
{\em Critical nonlinear Schr\"odinger Equation}, Non. Anal. TMA, {\bf 14} (1990), 807--836.

\bibitem{CCT}{M.~Christ, J. Colliander and T. Tao},
{\em Ill-posedness for nonlinear Schr\"odinger and wave equations}, preprint, arXiv:math/0311048v1 [math.AP].

\bibitem{CCT1}{M.~Christ, J. Colliander and T. Tao},
{\em Asymptotics, frequency modulation, and low regularity
 ill-posedness for canonical defocusing equations}, Amer. J. Math. {\bf 125} (2003), 1235--1293.

\bibitem{CGT}{J. Colliander, M. Grillakis and N. Tzirakis},
{\em Tensor products and correlation estimates with applications to nonlinear Schr\"odinger equations}, preprint.

\bibitem{CKSTT}{J. Colliander, M. Keel, G. Staffilani, H. Takaoka and T. Tao},
{\em Global well-posedness and scattering in the energy space for the critical nonlinear Schr\"odinger equation
in $\R^3$}, Annals of Math., to appear.

\bibitem{CIMM}{J. Colliander, S. Ibrahim, M. Majdoub and N. Masmoudi},
{\em Energy critical NLS in two space dimension}, preprint.

\bibitem{GSV}{J.~Ginibre, A.~Soffer, and G.~Velo},
{\em The global Cauchy problem for the critical nonlinear wave equation},
J. Funct. Anal. {\bf 110} (1992), 96--130.

\bibitem{GV1}{J. Ginibre and G. Velo},
{\em The global Cauchy problem for nonlinear Klein-Gordon equation},
Math. Z. {\bf 189} (1985), 487--505.

\bibitem{GV2}{J. Ginibre and G. Velo},
{\em Time decay of finite energy solutions of the non linear
Klein-Gordon and Schr\"odinger equations}, Ann. Inst. Henri.
Poincar\'e, {\bf 43} (1985), 399--442.

\bibitem{GV4} J.~Ginibre and G.~Velo,
{\it Scattering theory in the energy space for a class of nonlinear Schr\"odinger equations},
J. Math. Pures Appl. (9) {\bf 64} (1985), no.~4, 363--401.

\bibitem{GV3}
{J.~Ginibre, G.~Velo, G},
{\em Scattering theory in the energy space for a class of nonlinear wave equations}, Comm. Math. Phys. {\bf 123} (1989), 535--573.

\bibitem{Gr1} {M.~Grillakis},  {\em Regularity and asymptotic behaviour
of the wave equation with a critical nonlinearity},  Annals of Math.
{\bf 132} (1990), 485--509.

\bibitem{Gr2} {M.~Grillakis}, {\em Regularity for the wave equation with a critical
nonlinearity},  Comm. Pure Appl. Math. {\bf 46} (1992), 749--774.

\bibitem{2Dglobal}
S.~Ibrahim, M.~Majdoub and N.~Masmoudi,
{\em Global solutions for a semilinear, two-dimensional Klein-Gordon equation with exponential-type nonlinearity},
Comm.~Pure Appl.~Math. {\bf 59} (2006), no.~11, 1639--1658.

\bibitem{DlogSob}
S.~Ibrahim, M.~Majdoub and N.~Masmoudi,
{\em Double logarithmic inequality with a sharp constant}, Proc. Amer. Math. Soc. {\bf 135} (2007), no. 1, 87--97.

\bibitem{2Dinst}
S.~Ibrahim, M.~Majdoub and N.~Masmoudi,
{\it Ill-posedness of supercritical waves}, preprint.

\bibitem{IMM5}
S.~Ibrahim, M.~Majdoub and N.~Masmoudi,
{\it On the well-posedness of some 2D Klein-Gordon equations with exponential type nonlinearities}, preprint.

\bibitem{KOT}
H.~Kozono, T.~Ogawa and Y.~Taniuchi,
{\em The critical Sobolev inequalities in Besov spaces and regularity criterion to some semi-linear evolution equations},
Math. Z. {\bf 242} (2002), no. 2, 251--278.

\bibitem{LLT}
J.~F.~Lam, B.~Lippman, and F.~Tappert, {\em Self trapped laser beams in plasma}, Phys. Fluid {\bf20} (1977), 1176-1179.

\bibitem{Leb1} {G.~Lebeau}, {\em Nonlinear optics and supercritical
wave equation}, Bull. Soc. R. Sci. Li\`ege {\bf 70} (2001), No.~4-6, 267--306.

\bibitem{Leb2} {G.~Lebeau}, {\em Perte de r\'egularut\'e
pour l'\'equation des ondes surcritique}, Bull. Soc. Math. France {\bf 133} (2005), 145--157.

\bibitem{MP}
{N.~Masmoudi and F.~Planchon},
{\em On uniqueness for the critical wave equation},
Comm. Partial Differential Equations, {\bf 31} (2006), 1099--1107.

\bibitem{NO3}{M. Nakamura and T. Ozawa}, {\em Nonlinear Schr\"odinger
equations in the Sobolev space of critical order },
J. Funct. Anal. {\bf 155} (1998), 364--380.


\bibitem{NO1}
M. ~Nakamura and T. ~Ozawa, {\it Global solutions
in the critical Sobolev space for the wave equations with
nonlinearity of exponential growth}, Math. Z. {\bf 231} (1999), 479--487.

\bibitem{NO2}
M. ~Nakamura and T. ~Ozawa, {\it The Cauchy
problem for nonlinear wave equations in the Sobolev space of
critical order}, Discrete and Continuous Dynamical Systems, {\bf 5} (1999), no.~1, 215--231.

\bibitem{2Dsubcrit}
K.~Nakanishi, {\em Energy
scattering for nonlinear Klein-Gordon
and Schr\"odinger equations in spatial
dimensions $1$ and $2$},  J.~Funct.~Anal. {\bf 169} (1999), no.~1, 201--225.

\bibitem{3Dcrit}
K.~Nakanishi, {\em Scattering theory for the nonlinear Klein-Gordon equation with Sobolev critical power}, Internat.~Math.~Res.~Notices (1999), no.~1, 31--60.

\bibitem{P}
{F.~Planchon}, {\em On uniqueness for semilinear wave equations},
Math. Z. {\bf 244} (2003), 587--599.

\bibitem{PV}
{F.~Planchon and L.~Vega},
{\em Bilinear virial identities and applications}, preprint, arXiv:0712.4076v1 [math.AP].

\bibitem{Ruf}
{B.~Ruf}, {\em A sharp Trudinger-Moser type inequality for
unbounded domains in $\mathbb R\sp 2$.} J.~Funct.~Anal. {\bf 219}
(2005), no.~2, 340--367.

\bibitem{SS}{J.~Shatah and M.~Struwe}, {\em Well-posedness in the energy space for semilinear wave equation with critical growth},
IMRN, {\bf 7} (1994),  303--309.

\bibitem{Str}
{W.~A.~Strauss},
{\em On weak solutions of semi-linear hyperbolic equations},
An. Acad. Brasil. Ci. {\bf 42} (1970), 645--651.

\bibitem{Strauss}{W.~Strauss}, {\em Nonlinear wave equations},
Conf. Board of the Math. Sciences, {\bf 73}, Amer. Math. Soc., 1989.

\bibitem{Stru}
{M.~Struwe}, {\em Semilinear wave equations}, Bull. ~Amer.
~Math. ~Soc., {\bf N.S},  $n^\circ$ 26, 53-85, 1992.

\bibitem{Stru2}
{M.~Struwe},
{\em Uniqueness for critical nonlinear wave equations and wave maps via the energy inequality},
Comm. Pure Appl. Math. {\bf 52} (1999), 1179--1188.

\end{thebibliography}
\end{document}